\documentclass[10pt,reqno]{amsart}

\usepackage[utf8]{inputenc}
\usepackage[english]{babel}
\usepackage[dvipsnames]{xcolor}

\usepackage{hyperref}
\usepackage{xparse}% http://ctan.org/pkg/xparse
\NewDocumentCommand{\ceil}{s O{} m}{%
  \IfBooleanTF{#1} % starred
    {\left\lceil#3\right\rceil} % \ceil*[..]{..}
    {#2\lceil#3#2\rceil} % \ceil[..]{..}
}

\usepackage{enumerate}

\newtheorem{theorem}{Theorem}[section]
\newtheorem{lemma}[theorem]{Lemma}

\theoremstyle{remark}
\newtheorem{remark}[theorem]{\bf{Remark}}

\theoremstyle{definition}
\newtheorem{assumption}[theorem]{Assumption}

\newtheorem{definition}[theorem]{Definition}

\newcommand\cbrk{\text{$]$\kern-.15em$]$}}
\newcommand\opar{\text{\,\raise.2ex\hbox{${\scriptstyle
|}$}\kern-.34em$($}}
\newcommand\cpar{\text{$)$\kern-.34em\raise.2ex\hbox{${\scriptstyle |}$}}\,}

\newcommand{\aint}{-\hspace{-0.38cm}\int}

\newcommand\bL{\mathbb{L}}

\newcommand\bR{\mathbb{R}}
\newcommand\bH{\mathbb{H}}
\newcommand\bZ{\mathbb{Z}}

\newcommand\cB{\mathcal{B}}
\newcommand\cD{\mathcal{D}}
\newcommand\cF{\mathcal{F}}

\newcommand\cO{\mathcal{O}}
\newcommand\cS{\mathcal{S}}

\newcommand{\mysection}[1]{\section{#1}
\setcounter{equation}{0}}

\newcommand{\Ccinf}{C_{c}^{\infty}}
\newcommand{\R}{\mathbb{R}}

\begin{document}

\title[maximal regularity for time-fractional equations with weights]
{Weighted maximal $L_{q}(L_{p})$-regularity theory for time-fractional diffusion-wave equations with variable coefficients}

\author{Daehan Park}
\address{Stochastic Analysis and Application Research Center, Korea Advanced Institute of Science and Technology, 291 Daehak-ro, Yuseong-gu, Daejeon, 34141, Republic of Korea} \email{daehanpark@kaist.ac.kr}

\thanks{This work was supported by the National Research Foundation of Korea(NRF) grant funded by the Korea government(MSIT) (No. NRF-2019R1A5A1028324)} 



\begin{abstract}

We present a maximal $L_{q}(L_{p})$-regularity theory with Muckenhoupt weights for the equation
\begin{equation}\label{eqn 01.26.16:00}
\partial^{\alpha}_{t}u(t,x)=a^{ij}(t,x)u_{x^{i}x^{j}}(t,x)+f(t,x),\quad t>0,x\in\mathbb{R}^{d}.
\end{equation}
Here, $\partial^{\alpha}_{t}$ is the Caputo fractional derivative of order $\alpha\in(0,2)$ and $a^{ij}$ are functions of $(t,x)$. Precisely, we show that
\begin{equation*}
\begin{aligned}
&\int_{0}^{T}\left(\int_{\mathbb{R}^{d}}|(1-\Delta)^{\gamma/2}u_{xx}(t,x)|^{p}w_{1}(x)dx\right)^{q/p}w_{2}(t)dt
\\
&\quad \leq N \int_{0}^{T}\left(\int_{\bR^{d}}|(1-\Delta)^{\gamma/2}f(t,x)|^{p}w_{1}(x)dx\right)^{q/p}w_{2}(t)dt,
\end{aligned}
\end{equation*}
where $1<p,q<\infty$, $\gamma\in\bR$, and $w_{1}$ and $w_{2}$ are Muckenhoupt weights. This implies that we prove maximal regularity theory, and sharp regularity of solution according to regularity of $f$. To prove our main result, we also proved the complex interpolation of weighted Sobolev spaces,
$$
[H^{\gamma_{0}}_{p_{0}}(w_{0}), H^{\gamma_{1}}_{p_{1}}(w_{1})]_{[\theta]} = H^{\gamma}_{p}(w),
$$
where $\theta\in (0,1)$, $\gamma_{0},\gamma_{1}\in\bR$,  $p_{0},p_{1}\in(1,\infty)$, $w_{i}$ ($i=0,1$) are arbitrary $A_{p_{i}}$ weight, and
$$
\gamma=(1-\theta)\gamma_{0}+\theta\gamma_{1}, \quad \frac{1}{p}=\frac{1-\theta}{p_{0}} + \frac{\theta}{p_{1}},\quad w^{1/p}=w^{\frac{(1-\theta)}{p_{0}}}_{0}w^{\frac{\theta}{p_{1}}}_{1}.
$$
\end{abstract}

\maketitle

\mysection{Introduction}
The classical heat equation $\partial_{t}u=\Delta u$ describes the heat diffusion in homogeneous media. If we change the order of differentiation in time variable, then the equation becomes
\begin{equation*}
\partial_t^\alpha u(t,x) = \Delta u(t,x) +f(t,x),
\end{equation*}
and it describes diffusion of particles ($\alpha\in(0,1)$) or propagation of waves ($\alpha\in(1,2)$) in inhomogeneous media (see e.g. \cite{metzler1999anomalous,metzler2000random,mainardi1995fractional,mainardi2001fractional}).

In this article, we prove the unique solvability and maximal regularity theory for equation \eqref{eqn 01.26.16:00} with zero initial condition in Sobolev spaces with Muckenhoupt weights. Precisely, we prove that for any $p,q\in (1,\infty)$, $\gamma\in\bR$, $w_{1}=w_{1}(x)\in A_p(\bR^{d})$ and $w_{2}=w_{2}(t)\in A_q(\bR)$, there exists a unique solution satisfying

$$
\| |\partial^{\alpha}_{t}u|+|u|+|u_{x}|+|u_{xx}| \|_{\bH^{\gamma}_{q,p}(w_{2},w_{1},T)} \leq N \| f \|_{\bH^{\gamma}_{q,p}(w_{2},w_{1},T)},
$$
where the norm in $\bH^{\gamma}_{q,p}(w_{2},w_{1},T)$ is defined by
$$
\| f \|_{\bH^{\gamma}_{q,p}(w_{2},w_{1},T)} := \left( \int_0^T \left( \int_{\bR^d} |(1-\Delta)^{\gamma/2}f(t,x)|^p w_{1}(x)dx\right)^{q/p}w_{2}(t)dt\right)^{1/q},
$$
and $A_p(\bR^{d})$ denotes the class of Muckenhoupt $A_p$ weights defined on $\bR^{d}$.

The class of Muckenhoupt weights is the optimal class of weights for Hardy-Littlewood maximal functions to have boundedness in $L_{p}$ spaces with weight (see \cite{muckenhoupt1972,grafakos2009modern}). Also many types of Fourier multipliers role as a bounded operator in $L_{p}$ space with Muckenhoupt weight (see e.g. \cite{de1986calderon,kurtz1980littlewood,fackler2020}). These facts make it possible to develop theories for Sobolev spaces with Muckenhoupt weights (see e.g. \cite{gutierrez1991sobolev,miller1982sobolev,fabes1982local}). Also see e.g. \cite{goldshtein2009} for embedding theorems in Sobolev space with $A_{p}$ weight. In this article, we prove general interpolation inequalities in Sobolev spaces with Muckenhoupt weights (Theorem \ref{thm 09.24.13:17} and Remark \ref{rmk 01.14.15:40}). Precisely, we prove that for $\theta\in (0,1)$, $\gamma_{0},\gamma_{1}\in\bR$, $p_{0},p_{1}\in(1,\infty)$,   $w_{0}\in A_{p_{0}}$, $w_{1}\in A_{p_{1}}$, 
$$
\|u\|_{H^{\gamma}_{p}(w)}  \leq \|u\|^{1-\theta}_{H^{\gamma_{0}}_{p_{0}}(w_{0})}\|u\|^{\theta}_{H^{\gamma_{1}}_{p_{1}}(w_{1})},
$$
where
$$
\gamma=(1-\theta)\gamma_{0}+\theta\gamma_{1}, \quad \frac{1}{p}=\frac{1-\theta}{p_{0}} + \frac{\theta}{p_{1}},\quad w^{1/p}=w^{\frac{(1-\theta)}{p_{0}}}_{0}w^{\frac{\theta}{p_{1}}}_{1}.
$$
The above relation is widely studied in the literature. In \cite{L1982}, the author characterized complex interpolation spcaces of weighted Sobolev spaces when the given weight satisfies polynomial growth condition, which covers the results for unweighted Sobolev spaces in \cite{bergh2012interpolation}. For abstract function spaces with weight, complex interpolation was given in e.g. \cite{MS2012,LMV2018}, where the weights are types of power functions. For example, if $X$ is a UMD space, $w_{\mu}(t)=t^{(1-\mu)p}$ ($\mu\in(1/p,1])$, and $\gamma_{0},\gamma_{1}\geq0$, then
$$
[H^{\gamma_{0}}_{p}(\bR_{+},w_{\mu};X),H^{\gamma_{1}}_{p}(\bR_{+},w_{\gamma};X)]_{[\theta]} = H^{s}_{p}(\bR_{+},w_{\mu};X),
$$
where $H^{s}_{p}(\bR_{+},w_{\mu};X)$ is $X$-valued weighted Sobolev space defined on $\bR_{+}$. Here we emphasize that $w_{\mu}\in A_{p}(\bR)$ (see \cite[Lemma 2.8]{MS2012}).  Comparing the above-mentioned ones, our result only requires that the given weights are in $A_{p}$ class, not the exact growth or shape of weights. Also, when $\gamma_{0}=\gamma_{1}=1$, and $p_{0}=p_{1}$, then the condition for weights $w_{0}$ and $w_{1}$ was relaxed in \cite{CE2019} so that they need not be in $A_{p}$.

An $L_q(L_p)$-estimation without weight for abstract parabolic Volterra equations like
\begin{equation*}
\frac{\partial}{\partial t}\left(c_{0}u+\int_{-\infty}^{t}k_{1}(t-s)u(s,x)ds \right)=a^{ij}u_{x^{i}x^{j}}(t,x)+f(t,x)
\end{equation*}
was introduced in \cite{clement1992global, Pr1991} for $a^{ij}=\delta^{ij}$, $\alpha \in (0,1)$, and
$$
\frac{2}{\alpha q}+\frac{d}{p}<1,
$$
where $c_{0}\geq 0$ and $k_{1}(t)\geq ct^{-\alpha}$. The results of \cite{clement1992global, Pr1991} are based on semi group theory, and similar approach is used in \cite{zacher2005maximal} for general $a^{ij}(t,x)$ under the conditions that $p=q>1$, $a^{ij}$ are uniformly continuous in $(t,x)$, and
$$
\alpha\not\in \left\{\frac{2}{2p-1}, \frac{2}{p-1}-1,\frac{1}{p}, \frac{3}{2p-1}\right\}.
$$
For \eqref{eqn 01.26.16:00} the algebraic restriction on $\alpha,d,p$ and $q$ are dropped in \cite{kim17timefractionalpde} under the condition that $a^{ij}$ are uniformly continuous in $x$ and piecewise continuous in $t$. The result in \cite{kim17timefractionalpde} is obtained by using Calder\'on-Zygmund theorem, and estimation of fundamental solution. The condition for coefficients are significantly relaxed in \cite{dong2019lp} for the case $p=q$ and $\alpha\in (0,1)$ so that coefficients are measurable in $t$ and have small mean oscillation in $x$.

For weighted mixed norm estimate, see \cite{han19timefractionalAp} for the case $a^{ij}=\delta^{ij}$. The result of \cite{han19timefractionalAp} is based on solution representation by using kernel, and the authors control the mean oscillation of the solution and its derivative to find weighted mixed estimation. Quite recently, weighted mixed norm estimation under the same condition in \cite{dong2019lp} is proved in \cite{dong2021lp}. However, the results of \cite{dong2021lp,han19timefractionalAp} only cover the case $\gamma=0$. Also the result of \cite{han19timefractionalAp} only covers the case $a^{ij}=\delta^{ij}$, and the result of \cite{dong2021lp} does not cover the range of $\alpha\in(1,2)$.

Our result is an extension of \cite{dong2021lp,han19timefractionalAp} in the sense that we give general regularity theory in $L_{q}(L_{p})$ spaces with Muckenhoupt weights for equations with variable coefficients, and we cover the full range of $\alpha\in(0,2)$ without any restrictions on $\alpha,d,p,q$. Our approach depends on a perturbation argument used in \cite{kim17timefractionalpde,kry99analytic} which is a standard approach if one starts from the result of equations with constant coefficients. The main difficulty arises here, and we deal with this as follows.

First, since we consider maximal regularity theory for equations with variable coefficients in weighted space, we need to prove pointwise multiplier theorem like
$$
\int_{\bR^{d}}|(1-\Delta)^{k/2}(au)(x)|^{p}w_{1}(x)dx \leq N |a|^{p}_{C^{k}}\int_{\bR^{d}}|(1-\Delta)^{k/2}u(x)|^{p} w_{1}(x)dx.
$$
If $k$ is not a nonnegative integer, then we cannot obtain this by product rule. To handle this, we construct complex interpolation for Sobolev space with weight by using Fourier multiplier theorems and abstract interpolation theory. Second, the perturbation argument in \cite{kim17timefractionalpde,kry99analytic} uses uniform localization theorem for Sobolev space in \cite{kry94Littlewood}. Hence, we need to prove the corresponding theorem for Sobolev space with weight.

This article is organized as follows. In Section 2, we introduce basic definitions and facts related to fractional calculus, and Sobolev spaces with Muckenhoupt weights. Also, we present our main result, Theorem
\ref {theorem 5.1}. In Section 3, we prove properties of interpolation spaces of Sobolev spaces with weights, and uniform localization theorem for these spaces. In Section 4, we prove the main theorem.

We finish the introduction with notation used in this article. $\mathbb{R}^d$ denotes the $d$-dimensional Euclidean space of points $x = (x^1,\dots,x^d)$. We set $B_r(x) = \{ y\in\mathbb{R}^d:|x-y|<r \}$ and $B_r = B_r(0)$. If a set $E$ is in $\R^{d}$ (or $\R^{d+1}$), then $|E|$ is the Lebesgue measure of $E$. For $i=1,...,d$, multi-indices $\sigma=(\sigma_{1},...,\sigma_{d})$,
$\sigma_{i}\in\{0,1,2,...\}$, and functions $u(t, x)$ we set
$$
u_{x^{i}}=\frac{\partial u}{\partial x^{i}}=D_{i}u,\quad
D^{\sigma}u=D^{\sigma}_xu=D_{1}^{\sigma_{1}}\cdot...\cdot D^{\sigma_{d}}_{d}u.
$$
We also use the notation $D^m$ (or $D^m_x$) for partial derivatives of order $m$ with respect to $x$.
Similarly,
by $\partial_{t}^nu$ (or $\frac{d^n}{dt^n}u$) we
mean a partial derivative of order $n$ with respect to $t$. $C_c^\infty(\cO)$ denotes the collection of all infinitely differentiable functions with compact support in $\cO$, where $\cO$ is an open set in $\mathbb{R}^d$ or $\R^{d+1}$. Also $\cS=\cS(\bR^{d})$ denotes the space of Schwartz functions. For vector spaces $F$, and $G$ contained in a vector space $X$, we use the notation $F+G$ to denote the space of vectors $f+g$, where $f\in F$, and $g \in G$. For a measure space $(E,\mu)$, a Banach space $B$, and $p\in(1,\infty)$, $L_p(E,\mu;B)$ denotes the set of $B$-valued $\mu$-measurable functions $u$ on $E$ satisfying
\begin{equation*}
\| u \|_{L_p(E,\mu;B)} := \left(\int_E\|u\|_B^p d\mu\right)^{1/p}<\infty.
\end{equation*}
If $B = \R$, then $L_p(E,\mu;B) = L_p(E,\mu)$. For a measurble set $A$ and a measurable function $f$, we use the following notation
\begin{equation*}
\aint_{A}f(x)d\mu :=\frac{1}{\mu(A)}\int_{A}f(x)d\mu.
\end{equation*}
By $\cF$ and $\cF^{-1}$ we denote the Fourier and the inverse Fourier transform in $\bR^{d}$;
$$
\cF(f)=:\frac{1}{(2\pi)^{d/2}} \int_{\bR^{d}}e^{-i x\cdot \xi} f(x) dx, \quad \cF^{-1}(f)=:\frac{1}{(2\pi)^{d/2}} \int_{\bR^{d}}e^{i x\cdot \xi} f(\xi) d\xi.
$$
For two numbers $a,b\in(0,\infty)$ we write $a\sim b$ if there exists a constant $c>0$ independent of $a,b$ such that $c^{-1}a\leq b \leq ca$. Finally, if we write $N=N(a,b,\ldots)$, then this means that the constant $N$ depends only on $a,b,\ldots$.

\mysection{Main result}

We first introduce some definitions and facts related to the fractional calculus. For more details, see e.g. \cite{baleanu12fractional,podlubny98fractional,richard14fractional,samko93fractional}.
For $\alpha>0$ and $\varphi\in L_1((0,T))$, the Riemann-Liouville fractional integral  of order $\alpha$ is defined by
$$ I^\alpha_{t} \varphi(t) :=(I^\alpha_{t} \varphi) (t):= \frac{1}{\Gamma(\alpha)}\int_0^t(t-s)^{\alpha-1}\varphi(s)ds,\quad  t\leq T.
$$
By Jensen's inequality, for any $p\in[1,\infty]$, 
\begin{equation*}
\|I^\alpha_t \varphi \|_{L_p((0,T))}\leq N(\alpha,p,T)\| \varphi \|_{L_p((0,T))}.
\end{equation*}
It is also easy to check that  if $\varphi$ is bounded then $I^{\alpha}_t \varphi(t)$ is a continuous function satisfying $I^{\alpha}_t\varphi(0)=0$.   

Let $n$ be the integer such that $n-1\leq \alpha<n$. If $\varphi$ is $(n-1)$-times differentiable, and $(\frac{d}{dt})^{n-1}I_{t}^{n-\alpha}\varphi$ is absolutely continuous on $[0,T]$, then the Riemann-Liouville fractional derivative $D_t^\alpha \varphi$ and  the Caputo fractional derivative $\partial_t^\alpha\varphi$ are defined as follows.
\begin{equation*}
D^{\alpha}_t \varphi (t):=(D^{\alpha}_t \varphi) (t):=(I_{t}^{n-\alpha}\varphi)^{(n)}(t),
\end{equation*}
\begin{equation}\label{Caputo_Riemann}
\partial^{\alpha}_{t}\varphi (t) :=D^{\alpha}_{t}\left(\varphi(t)-\sum_{k=0}^{n-1}\frac{t^{k}}{k !}\varphi^{(k)}(0)\right)(t).
\end{equation}
Obviously, $D^\alpha_t\varphi = \partial_t^\alpha \varphi$  if $\varphi(0) = \varphi^{(1)}(0) = \cdots = \varphi^{(n -1)}(0) = 0$.  
It is easy to show that, for any $\alpha,\beta \geq0$, 
$$
I^{\alpha+\beta}_{t}\varphi(t)=I^{\alpha}_{t}I^{\beta}_{t}\varphi(t), \quad
D^\alpha_{t} D^\beta_{t}\varphi = D^{\alpha+\beta}\varphi,
$$
and
$$ 
D^\alpha_{t} I^\beta_{t} \varphi = \begin{cases} D^{\alpha-\beta}_{t}\varphi &\mbox{ if }\alpha>\beta \\ I^{\beta-\alpha}_{t}\varphi &\mbox{ if }\alpha\leq\beta\end{cases}
$$
Furthermore, if $\varphi$ is sufficiently smooth (say, $\varphi\in C^{n}([0,T]))$  and $\varphi(0)=\cdots =\varphi^{(n-1)}(0)=0$, then
\begin{equation}
    \label{eqn 10.1}
I^{\alpha}_t\partial^{\alpha}_t \varphi(t): = I^{\alpha}_t (\partial^{\alpha}_t \varphi) (t)=\varphi(t), \quad \forall \, t\in [0,T].
\end{equation}
Consequently, if $\varphi\in C^{2}([0,T])$ and $\alpha\in (0,2)$, then $\partial^{\alpha}_t \varphi=f$  is equivalent to
\begin{equation}
   \label{eqn 4.29.5}
   \varphi(t)-\varphi(0)-1_{\alpha>1}\varphi'(0)t=I^{\alpha}_t f(t), \quad \forall\,  t\in [0,T].
   \end{equation}

Now we introduce the class of weights used in this article.
\begin{definition}\label{def 05.06.1}
Let $1<p<\infty$. We write $w\in A_p(\mathbb{R}^{d})$ if $w(x)$ is a nonnegative locally integrable function on $\bR^d$ such that
\begin{equation*}
[w]_{p}:=\sup_{B}\left(\aint_{B}w(x)dx\right)\left(\aint_{B}w(x)^{-1/(p-1)}dx\right)^{p-1}<\infty,
\end{equation*}
where the supremum is taken over all balls $B$ in $\mathbb{R}^{d}$, and
\begin{equation*}
\aint_{B}w(x)dx=\frac{1}{|B|}\int_{B}w(x)dx.
\end{equation*}
If $w\in A_{p}(\mathbb{R}^{d})$, then $w$ is said to be an $A_p$ weight.
\end{definition}

\begin{remark} \label{rmk 06.20.1}
The necessary and sufficient condition for Hardy-Littlewood maximal function 
to be bounded in $L_p(w dx)$ is $w\in A_p(\bR^{d})$ (see \cite{muckenhoupt1972} or \cite[Theorem 9.1.9]{grafakos2009modern}). Therefore, if one uses an approach based on   sharp and maximal functions, then it is natural to consider $L_p$-spaces with $A_p$ weights for full generality.
\end{remark}

\begin{remark}\label{rmk 09.27.21:52}
(i) From the definition of $[w]_{p}$, we can easily check that for any $c>0$, $y\in\bR^{d}$ and $w\in A_{p}(\mathbb{R}^{d})$, the functions $w^{c}(x)=w(cx)$, and $w_{y}(x)=w(x-y)$ are in $A_{p}(\mathbb{R}^{d})$ and $[w^{c}]_{p}=[w_{y}]_{p}=[w]_{p}$.

Also for $w\in A_{p}(\bR^{d})$ and for any $d\times d$ orthogonal matrix $Q$, it is easy to see that $w_{Q}(x)=w(Qx)\in A_{p}(\bR^{d})$, and $[w_{Q}]_{p}=[w]_{p}$. Therefore, for any symmetric $d\times d$ matirx $\Sigma$ with positive eigenvalues, one can check that $w_{\Sigma}(x)=w(\Sigma x)\in A_{p}(\bR^{d})$ and $[w]_{p}\sim [w_{\Sigma}]_{p}$.

(ii) The class $A_{p}$  is increasing as $p$ increases, and it holds that
\begin{equation*}
A_{p}(\mathbb{R}^{d})=\bigcup_{q\in(1,p)} A_{q}(\mathbb{R}^{d}).
\end{equation*}
More precisely, for any $w\in A_{p}(\mathbb{R}^{d})$, one can find $q<p$, which depends on $d,p$, and $[w]_{p}$ such that $w\in A_{q}(\mathbb{R}^{d})$.
\end{remark}

Let $p,q\in(1,\infty)$, $\gamma\in\bR$, and $T\in(0,\infty]$. For $w_{1} = w_{1}(x)\in A_{p}(\mathbb{R}^{d})$ and $w_{2}=w_{2}(t)\in A_{q}(\mathbb{R})$, we define
$$
L_{p}(w_{1})=L_{p}(\R^{d},w_{1}dx),
$$
and we denote $H^{\gamma}_{p}(w_{1})$ by the class of tempered distributions satisfying
$$
\|u\|_{H^{\gamma}_{p}(w_{1})} = \|(1-\Delta)^{\gamma/2}u\|_{L_{p}(w_{1})} <\infty,
$$
where the operator $(1-\Delta)^{\gamma/2}$ is defined by
$$
(1-\Delta)^{\gamma/2} u = \cF^{-1} \left((1+|\xi|^2)^{\gamma/2}\cF u \right).
$$
The action of $u\in H^{\gamma}_{p}(w_{1})$ on $\phi\in\cS(\R^{d})$, which is denoted by $(u,\phi)$, is defined by
\begin{equation}
   \label{eqn 5.12.2}
(u,\phi):=((1-\Delta)^{\gamma/2}u, (1-\Delta)^{-\gamma/2}\phi):=\int_{\bR^{d}}(1-\Delta)^{\gamma/2}u (1-\Delta)^{-\gamma/2}\phi dx.
\end{equation}
Also we define
$$
\bL_{q,p}(w_{2},w_{1},T)=L_{q}((0,T), w_{2}dt ; L_{p}(w_{1})),
$$
and
$$
\bH^{\gamma}_{q,p}(w_{2},w_{1},T)=L_{q}((0,T),w_{2}dt ; H^{\gamma}_{p}(w_{1})).
$$
We omit $T$ if $T = \infty$. For example,
\begin{equation*}
\bL_{q,p}(w_{2},w_{1}) = \bL_{q,p}(w_{2},w_{1},\infty) = L_q((0,\infty),w_{2}dt ;L_p(w_{1})).
\end{equation*}

The norms of these function spaces are defined in a natural way. For example,
\begin{equation*}
\begin{aligned}
\|f\|_{\bL_{q,p}(w_{2},w_{1},T)}&:=\left(\int_0^T \|f(t,\cdot)  \|_{L_p(w_{1})}^qw_{2}(t) dt\right)^{1/q} \\
&=\left(\int_{0}^{T}\left( \int_{\R^{d}}|f(t,x)|^{p}w_{1}(x)dx \right)^{q/p}w_{2}(t)dt\right)^{1/q}.
\end{aligned}
\end{equation*}

\begin{remark}\label{rmk 05.08.1}

(i) Suppose that $w_{1}\in A_{p}(\bR^{d})$. Then the weighted Sobolev spaces have similar properties as the usual $L_p$ spaces. For example, one can check that $C_c^\infty(\R^{d})$ is dense in $H_p^{\gamma}(w_{1})$ and $H_p^{\gamma}(w_{1})$ is a Banach space. The dual space of $H^{\gamma}_{p}(w_{1})$ is $H^{-\gamma}_{p'}(\tilde{w_{1}})$, where
\begin{equation} \label{dual}
1/p+1/p'=1,\quad \tilde{w_{1}}=w^{-\frac{1}{p-1}}_{1}.
\end{equation}
Also, for any $\gamma,\nu\in\mathbb{R}$, $(1-\Delta)^{\nu/2}$ is an isometry from $H^{\gamma}_{p}(w_{1})$ to $H^{\gamma-\nu}_{p}(w_{1})$, and $H^{\nu}_{p}(w_{1})\subset H^{\gamma}_{p}(w_{1})$ if $\nu\leq \gamma$. Finally, if $\gamma$ is a nonnegative integer, then for any multi-index $\sigma$ with $|\sigma|\leq \gamma$, the operator
$$
D^{\sigma}(1-\Delta)^{-\gamma/2} \quad |\sigma|\leq \gamma,
$$ 
is a bounded operator in $L_{p}(w)$ (see \cite{miller1982sobolev}). This implies that
$$
\|u\|_{H^{\gamma}_{p}(w_{1})}\sim \sum_{|\beta| \leq \gamma} \|D^{\beta}u\|_{L_{p}(w_{1})}.
$$

(ii) Let $\phi \in \mathcal{S}(\bR^{d})$ and $\phi_{\varepsilon}(x):=\varepsilon^{-d}\phi(x/\varepsilon)$, $\varepsilon\in(0,\infty)$.  Then by \cite[Remark 2.4]{han19timefractionalAp}, for any $f\in L_{p}(w_{1})$, it follows that $f^{\varepsilon}(x):=\int_{\mathbb{R}^{d}} \phi_{\varepsilon}(x-y)f(y)dy$ satisfies
$$
\|f^{\varepsilon}\|_{L_{p}(w_{1})} \leq N \|f\|_{L_{p}(w_{1})},
$$
where the constant $N$ is independent of $\varepsilon,f$, and depends only on $d,p,\|\phi\|_{L_{\infty}},\|D\phi\|_{L_{\infty}}$, and $[w_{1}]_{p}$. Moreover, if we take $\phi\in \Ccinf(B_{1})$, then $f^{\varepsilon}$ converges to $f$ in $L_{p}(w_{1})$ as $\varepsilon\downarrow 0$. Similarly, for $f\in H^{\gamma}_{p}(w_{1})$, and $\phi\in \Ccinf(B_{1})$, the function $f^{\varepsilon}(x):=(f,\phi_{\varepsilon}(x-\cdot))$ is infinitely differentiable and converges to $f$ in $H^{\gamma}_{p}(w_{1})$ as $\varepsilon\downarrow 0$. Indeed, if we let $(1-\Delta)^{\gamma/2}f=h$,
\begin{equation*}
\begin{aligned}
\|f-f^{\varepsilon}\|_{H^{\gamma}_{p}(w_{1})} &= \|h-h^{\varepsilon}\|_{L_{p}(w_{1})} 
\\
&\leq \|h-\eta \|_{L_{p}(w_{1})} + \|\eta-\eta^{\varepsilon}\|_{L_{p}(w_{1})} + \|\eta^{\varepsilon}-h^{\varepsilon}\|_{L_{p}(w_{1})}
\\
&\leq N \delta +  \|\eta-\eta^{\varepsilon}\|_{L_{p}(w_{1})},
\end{aligned}
\end{equation*}
where $\delta>0$, and $\eta\in \Ccinf(\bR^{d})$ such that $\|\eta-h\|_{L_{p}(w_{1})}<\delta$. By letting $\varepsilon \downarrow 0$, since $\delta>0$ is arbitrary we have the desired result.
\end{remark}

\begin{lemma}\label{lem 12.21.16:37}
Let $\alpha>0$, $1<q<\infty$, $0<T<\infty$,  and let $w_{2}\in A_{q}(\mathbb{R})$. Then for any $f\in L_{q}((0,T),w_{2}dt)$, we have
\begin{equation}\label{eqn 12.22.16:40}
\|I^{\alpha}_{t}f\|_{L_{q}((0,T),w_{2}dt)} \leq  T^{\alpha} N  \|f\|_{L_{q}((0,T),w_{2}dt)}
\end{equation}
where the constant $N$ depends only on $\alpha,q,[w_{2}]_{q}$.
\end{lemma}

\begin{proof}
For the case $\alpha\in(0,1)$,  see \cite[Lemma 5.5 (a)]{dong2021lp}. Now assume that $\alpha\geq 1$. Then for each $t\leq T$, we have
\begin{equation}\label{eqn 01.28.14:37}
\begin{aligned}
|I^{\alpha}_{t}f(t)| &\leq \frac{1}{\Gamma(\alpha)}\int_{0}^{t}|t-s|^{\alpha-1} |f(s)| ds
\\
& \leq \frac{1}{\Gamma(\alpha)} \int_{0}^{t} |t|^{\alpha-1}|f(s)|ds \leq \frac{T^{\alpha}}{\Gamma(\alpha)} \frac{2}{2t}\int_{0}^{2t}|f(s)|ds
\\
& \leq \frac{2T^{\alpha}}{\Gamma(\alpha)} \mathcal{M}f (t),
\end{aligned}
\end{equation}
where $\mathcal{M}f$ is the Hardy-Littlewood maximal function defined as
$$
\mathcal{M}f(t)=\sup_{t\in(a,b)}\aint_{a}^{b}|f(s)|ds.
$$
Extend $f=0$ ouside of $(0,T)$ and take $\phi_{n}\in \Ccinf(\bR)$ such that $0\leq \phi_{n} \leq 1$, $\phi(t)=1$ for $0\leq t \leq T$, and $\phi_{n}(t)=0$ for $t\leq -1/n$ or $t\geq T+1/n$. Since $f\phi_{n}=f$ on $(0,T)$, by \eqref{eqn 01.28.14:37} and Remark \ref{rmk 06.20.1}, we have
\begin{equation*}
\begin{aligned}
\|I^{\alpha}_{t}f\|_{L_{q}((0,T),w_{2}dt)}&=\|I^{\alpha}_{t}(f\phi_{n})\|_{L_{q}((0,T),w_{2}dt)}
\\
& \leq T^{\alpha} N \|\mathcal{M}(f\phi_{n})\|_{L_{q}((0,T),w_{2}dt)} \leq T^{\alpha} N \|\mathcal{M}(f\phi_{n})\|_{L_{q}(\bR,w_{2}dt)}
\\
&\leq T^{\alpha} N \|f\phi_{n}\|_{L_{q}(\bR,w_{2}dt)} 
\\
&\leq  T^{\alpha} N \|f\phi_{n}\|_{L_{q}((-1/n,T+1/n),w_{2}dt)}.
\end{aligned}
\end{equation*}
Letting $n\to \infty$, we get the desired result. The lemma is proved.
\end{proof}

Now we introduce our solution space and related facts.

\begin{definition}\label{def 05.07.2}
Let $0<\alpha<2$, $1<p,q<\infty$, $w_{1}\in A_{p}(\R^{d})$, $w_{2}\in A_{q}(\R)$, $\gamma\in \mathbb{R}$ and $T<\infty$. 

(i) We write $u\in \bH^{\alpha,\gamma+2}_{q,p}(w_{2},w_{1},T)$ if there exists a sequence $u_{n}\in C^{\infty} ([0,\infty)\times\R^{d})$ such that $u_{n}$ converges to $u$ in $\bH^{\gamma+2}_{q,p}(w_{2},w_{1},T)$, and $\partial^{\alpha}_{t}u_{n}$ is a Cauchy sequence in $\bH^{\gamma}_{q,p}(w_{2},w_{1},T)$. In this case, we say that $u_{n}$ is a defining sequence of $u$ and we write 
$$
\partial^{\alpha}_{t}u=\lim_{n\to\infty}\partial^{\alpha}_{t}u_{n}.
$$

(ii) For $u\in \bH^{\alpha,\gamma+2}_{q,p}(w_{2},w_{1},T)$, we write  $u\in \bH^{\alpha,\gamma+2}_{q,p,0}(w_{2},w_{1},T)$ if there is a defining sequence $u_{n}\in C^{\infty}([0,\infty)\times\R^{d})$ such that 
$$
u_n(0,x)=1_{\alpha>1}\partial_{t}u_{n}(0,x)=0.
$$
\end{definition}

\begin{lemma}\label{lem 05.08.1}
Let $1<p,q<\infty$, $w_{1}\in A_{p}(\R^{d})$, $w_{2}\in A_{q}(\R)$, $\gamma\in\bR$, and $T<\infty$.
 
(i) The spaces $\bH^{\alpha,\gamma+2}_{q,p}(w_{2},w_{1},T)$ and  $\bH^{\alpha,\gamma+2}_{q,p,0}(w_{2},w_{1},T)$ are Banach spaces with respect to the norm
\begin{equation*}
\|u\|_{\bH^{\alpha,\gamma+2}_{q,p}(w_{2},w_{1},T)}:=\|u\|_{\bH^{\gamma+2}_{q,p}(w_{2},w_{1},T)}+\|\partial^{\alpha}_{t}u\|_{\bH^{\gamma}_{q,p}(w_{2},w_{1},T)}.
\end{equation*}

(ii) Let $\nu\in\bR$. Then the operator $(1-\Delta)^{\nu/2}$ is an isometry from $\bH^{\alpha,\gamma+2}_{q,p}(w_{2},w_{1},T)$ to $\bH^{\alpha,\gamma-\nu+2}_{q,p}(w_{2},w_{1},T)$, and from $\bH^{\alpha,\gamma+2}_{q,p,0}(w_{2},w_{1},T)$ to $\bH^{\alpha,\gamma-\nu+2}_{q,p,0}(w_{2},w_{1},T)$ respectively. Moreover, $\partial^{\alpha}_{t}(1-\Delta)^{\nu/2}u=(1-\Delta)^{\nu/2}\partial^{\alpha}_{t}u$.

(iii) $C^{\infty}_c((0,\infty)\times \bR^d)$ is dense in $\bH^{\alpha,\gamma+2}_{q,p,0}(w_{2},w_{1},T)$. 

(iv) Suppose that $u\in\mathbb{H}^{\alpha,\gamma+2}_{q,p,0}(w_{2},w_{1},T)$, and let $\partial^{\alpha}_{t}u=f$. Then for any $t\leq T$ we have
\begin{equation}\label{eqn 12.21.16:05}
\|u\|^{q}_{\mathbb{H}^{\gamma}_{p}(w_{2},w_{1},t)} \leq t^{\alpha q} N  \|f\|^{q}_{\mathbb{H}^{\gamma}_{q,p}(w_{2},w_{1},t)}
\end{equation}
where the constant $N$ depends only on $\alpha,q,[w_{2}]_{q}$.
\end{lemma}

\begin{proof}
(i) It can be readily proved by following a straightforward argument. 

(ii) Suppose that $u\in \mathbb{H}^{\gamma+2}_{q,p}(w_{2},w_{1},T)$, and let $u_{n}$ be a defining sequence of $u$. It is easy to see that $(1-\Delta)^{\nu/2}u_{n}$ converges to $(1-\Delta)^{\nu/2}u$ in $\mathbb{H}^{\gamma-\nu+2}_{q,p}(w_{2},w_{1},T)$, and $\partial^{\alpha}_{t}(1-\Delta)^{\nu/2}u_{n}=(1-\Delta)^{\nu/2}\partial^{\alpha}_{t}u_{n}$ converges to $(1-\Delta)^{\nu/2}\partial^{\alpha}_{t}u$ in $\mathbb{H}^{\gamma-\nu}_{q,p}(w_{2},w_{1},T)$. Let $(1-\Delta)^{\nu/2}u_{n}=v_{n}$, and define $(v_{n})^{\varepsilon}$ as in Remark \ref{rmk 05.08.1} (ii). Then it follows that $(v_{n})^{\varepsilon}\in C^{\infty}([0,\infty)\times\bR^{d})$, $(\partial^{\alpha}_{t}v_{n})^{\varepsilon}=\partial^{\alpha}_{t}(v_{n})^{\varepsilon}$, and $(v_{n})^{\varepsilon}$ converges to $v_{n}$ in $\mathbb{H}^{\gamma-\nu+2}_{q,p}(w_{2},w_{1},T)$ as $\varepsilon\downarrow 0$. Therefore, by taking suitable sequence $\varepsilon_{n}>0$ such that $\varepsilon_{n} \downarrow 0$,  we have $(v_{n})^{\varepsilon_{n}}$ converges to $(1-\Delta)^{\nu/2}u$ in $\mathbb{H}^{\alpha,\gamma-\nu+2}_{q,p}(w_{2},w_{1},T)$, and $\partial^{\alpha}_{t}(1-\Delta)^{\nu/2}u=(1-\Delta)^{\nu/2}\partial^{\alpha}_{t}u$ by definition of $\mathbb{H}^{\alpha,\gamma-\nu+2}_{q,p}(w_{2},w_{1},T)$. Finally, from the above argument, we have
\begin{equation*}
\begin{aligned}
\|(1-\Delta)^{\nu/2}u\|_{\mathbb{H}^{\alpha,\gamma-\nu+2}_{q,p}(w_{2},w_{1},T)} &= \lim_{n\to\infty}\|(v_{n})^{\varepsilon_{n}}\|_{\mathbb{H}^{\alpha,\gamma-\nu+2}_{q,p}(w_{2},w_{1},T)} 
\\
&= \lim_{n\to\infty}\|u_{n}\|_{\mathbb{H}^{\alpha,\gamma+2}_{q,p}(w_{2},w_{1},T)}
\\
&= \|u\|_{\mathbb{H}^{\alpha,\gamma+2}_{q,p}(w_{2},w_{1},T)}.
\end{aligned}
\end{equation*} 
The assertion for $u\in \mathbb{H}^{\alpha,\gamma+2}_{q,p,0}(w_{2},w_{1},T)$ can be proved similarly by adding the condition $u_{n}(0,x)=1_{\alpha>1}\partial_{t}u_{n}(0,x)=0$.

(iii) If $\gamma=0$, then the statement is proved in \cite[Lemma 2.6 (ii)]{han19timefractionalAp}. For general case suppose that $u\in\mathbb{H}^{\alpha,\gamma+2}_{q,p,0}(w_{2},w_{1},T)$ and let $v=(1-\Delta)^{\gamma/2}u$. Then by (ii), we have $v\in \mathbb{H}^{\alpha,2}_{q,p,0}(w_{2},w_{1},T)$. Take $v_{n}\in\Ccinf((0,\infty)\times\bR^{d})$ which converges to $v$ in $\mathbb{H}^{\alpha,2}_{q,p,0}(w_{2},w_{1},T)$, and let $\bar{u}_{n}=(1-\Delta)^{-\gamma/2}v_{n}$. Then $\bar{u}_{n}\in C^{\infty}((0,\infty)\times\bR^{d})$, and $\bar{u}_{n}$ converges to $u$ in $\mathbb{H}^{\alpha,\gamma+2}_{q,p}(w_{2},w_{1},T)$. Moreover, for each $n$, $\bar{u}_{n}\in \mathbb{H}^{\alpha,m}_{q,p}(w_{2},w_{1},T)$ for any $m=1,2\dots$.  Therefore, since $\bar{u}_{n}$ has a compact support in $t$, if we take a function $\zeta\in\Ccinf(\bR^{d})$ such that $\zeta=1$ on $B_{1}$, and $\zeta=0$ outside of $B_{2}$, then $u_{n,k}(t,x)=\zeta(x/k)\bar{u}_{n}(t,x)\in\Ccinf((0,\infty)\times\bR^{d})$ converges to $\bar{u}_{n}$ in $\mathbb{H}^{\alpha,m}_{q,p}(w_{2},w_{1})$ for any $m=1,2,\dots$ as $k\to\infty$. Therefore, by taking appropriate subsequence $k_{n}\to \infty$ (recall Remark \ref{rmk 05.08.1} (i)), we have the desired result.

(iv) Take $u_{n}\in\Ccinf((0,\infty)\times\bR^{d})$ which converges to $u$ in $\mathbb{H}^{\alpha,\gamma+2}_{q,p}(w_{2},w_{1},T)$, and let $\partial^{\alpha}_{t}u_{n}=f_{n}$. First, note that
$$
(1-\Delta)^{\gamma/2}u_{n}(s,x)=N(\alpha)\int_{0}^{s}(s-r)^{\alpha-1}(1-\Delta)^{\gamma/2}f_{n}(r,x)dr
$$
for all $(s,x)\in[0,\infty)\times\bR^{d}$. Applying Minkowski's inequality, we have
\begin{equation*}
\begin{aligned}
\|u_{n}(s,\cdot)\|_{H^{\gamma}_{p}(w_{1})} \leq N \int_{0}^{s} (s-r)^{\alpha-1}\|f_{n}(r,\cdot)\|_{H^{\gamma}_{p}(w_{1})} dr \quad \forall \, s\in [0,\infty),
\end{aligned}
\end{equation*}
where the constant $N$ does not depend on $s$. By \eqref{eqn 12.22.16:40}, we have
\begin{equation*}
\begin{aligned}
\|u_{n}\|^{q}_{\mathbb{H}^{\gamma}_{p}(w_{2},w_{1},t)} &= N \int_{0}^{t} \left( \int_{0}^{s} (s-r)^{\alpha-1}\|f_{n}(r,\cdot)\|_{H^{\gamma}_{p}(w_{1})} dr \right)^{q} w_{2}(s)ds
\\
&= N \|I^{\alpha}_{t}(\|f_{n}\|_{H^{\gamma}_{p}(w_{1})})(\cdot)\|^{q}_{L_{q}((0,t),w_{2}dt)}
\\
&= N t^{\alpha  q} \|f_{n}\|^{q}_{\mathbb{H}^{\gamma}_{q,p}(w_{2},w_{1},t)},
\end{aligned}
\end{equation*}
where the constant $N$ depends only on $\alpha,q,[w_{2}]_{q}$. Hence, by letting $n\to\infty$. we have \eqref{eqn 12.21.16:05}. The lemma is proved. 
\end{proof}

\begin{remark}\label{rmk 01.14.13:13}
Note that the construction of $v_{n}$ in Lemma \ref{lem 05.08.1} (iii) comes from \cite[Theorem 2.7 (iii)]{kim17timefractionalpde}, and from this one can observe that functions $v_{n}\in\Ccinf((0,\infty)\times\bR^{d})$ are given by
$$
\eta(t)\eta_{3}(c_{n}x)\int_{0}^{\infty}\eta_{1}((t-s)/a_{n})\int_{\bR^{d}}((1-\Delta)^{\gamma/2}u(s,y))\eta_{2}((x-y)/b_{n})dyds,
$$
where $\eta\in C^{\infty}([0,\infty))$ such that $\eta(t)=1$ for $t\leq T$ and $\eta(t)=0$ for all $t>T+1$,
$$
\eta_{1}\in \Ccinf((1,2)), \quad \eta_{2},\eta_{3}\in\Ccinf(\bR^{d}),
$$
and $a_{n},b_{n},c_{n}>0$ are sequences converge to $0$. Therefore, if we have $u,v\in \mathbb{H}^{\alpha,\gamma+2}_{q,p,0}(w_{2},w_{1},T)$ such that $u(t)=v(t)$ for $t<T_{0}$ for some $T_{0}<T$, then one can take a defining sequences $u_{n}$ of $u$ and $v_{n}$ of $v$ in $\Ccinf((0,\infty)\times\bR^{d})$ such that 
$$
u_{n}(t)=v_{n}(t) \quad \text{in} \quad H^{\gamma+2}_{p}(w_{1}) \quad \forall\, t<T_{0}+a_{n}.
$$
\end{remark}

To introduce assmuptions on the coefficients, for each $r>0$ take $\kappa'\in[0,1)$ as follows;
\begin{enumerate}[(i)]
\item
if $r=1,2,\dots,$ then take $\kappa'=0$,
\item
if $r\neq 1,2,\dots$, then take $\kappa'\in(0,1)$ such that $r+\kappa'\neq 1,2,\dots$. 
\end{enumerate}
For $r\geq0$ set
\begin{equation}\label{eqn 01.06.14:42}
B^{r}=\begin{cases} L_{\infty}(\R^{d}) &\mbox{if}\quad r=0
\\
C^{r-1,1}(\R^{d}) & \mbox{if}\quad r=1,2,3,\dots
\\
C^{r+\kappa'}(\R^{d})&\mbox{otherwise},
\end{cases}
\end{equation}
where $C^{r+\kappa'}(\R^{d})=C^{r+\kappa'}$ is H\"older space and $C^{r-1,1}(\R^{d})=C^{r-1,1}$ is the space of $(r-1)$-times differentiable functions whose $(r-1)$-th derivatives are Lipschitz continuous. 

\begin{remark}\label{rmk 01.06.14:48}
Let $r>0$ and let $r=[r]_{-}+\{r\}^{+}$, where $0\leq [r]_{-}<r$ is an integer and $0<\{r\}^{+}\leq 1$. For example, if $r=1$, then $[r]_{-}=0$, and $\{r\}^{+}=1$. We define Zygmund space $\mathcal{C}^{r}(\bR^{d})=\mathcal{C}^{r}$, which consists of continuous functions $h$ satisfying
\begin{equation*}
\begin{aligned}
\|h\|_{\mathcal{C}^{r}} &:= \sum_{0\leq |\beta|\leq [r]_{-}} \sup_{x\in\bR^{d}}|D^{\beta}h(x)|
\\
&\quad + \sum_{|\beta|=[r]_{-}} \sup_{y\neq 0,y\in \bR^{d}}\frac{|D^{\beta}h(x+2y)-2D^{\beta}h(x+y)+D^{\beta}h(x)|}{|y|^{\{r\}^{+}}} <\infty.
\end{aligned}
\end{equation*}
It is well-known that if  $r\neq1,2,\dots$, then $\mathcal{C}^{r}$, and $C^{r}$ are equivalent, and if $r=1,2\dots$, then $\mathcal{C}^{r}$ is broader space than $C^{r}$. Moreover, we can easily check that for any $\varepsilon>0$, $\mathcal{C}^{r+\varepsilon}\subset C^{r-1,1}$. Finally, note that $\mathcal{C}^{r}=B^{r}_{\infty,\infty}$ for all $r>0$, where $B^{r}_{\infty,\infty}$ is Besov space (see e.g. \cite[Section 2.3.5]{triebel1983}).
\end{remark}

\begin{assumption}\label{asm 07.16.1}

(i) $a^{ij}(t,x),b^{i}(t,x),c(t,x)$ are  $\cB(\R^{d+1})$-measurable functions.

(ii) There exists a constant $0<\delta<1$ so that for any $(t,x)$
\begin{equation}\label{eqn 07.22.2}
\begin{gathered}
\delta|\xi|^{2}\leq a^{ij}(t,x)\xi^{i}\xi^{j}\leq \delta^{-1}|\xi|^{2},\quad \forall\xi\in\R^{d}.
\end{gathered}
\end{equation}

(iii) The coefficients $a^{ij}(t,x)$ is uniformly continuous in $(t,x)$.

(iv) For each $t$,
\begin{equation*}
|a^{ij}(t,\cdot)|_{B^{|\gamma|}} + |b^{i}(t,\cdot)|_{B^{|\gamma|}}+ |c(t,\cdot)|_{B^{|\gamma|}} \leq \delta^{-1}.
\end{equation*}

(v) If $p\neq q$, then $\lim_{|x|\to\infty}a^{ij}(t,x)$ exists uniformly in $t\in(0,T)$.

\end{assumption}

Finally, we introduce our main result.

\begin{theorem} \label{theorem 5.1}
Let $0<\alpha<2$, $1<p,q<\infty$, $\gamma\in\bR$, $w_{1}=w_{1}(x)\in A_{p}(\R^{d})$, $w_{2}=w_{2}(t)\in A_{q}(\R)$, and $T<\infty$. Suppose that Assumption \ref{asm 07.16.1} holds. Then for any $f\in \bH^{\gamma}_{q,p}(w_{2},w_{1},T)$, there exists a unique solution $u\in\mathbb{H}^{\alpha,\gamma+2}_{q,p,0}(w_{2},w_{1},T)$ for the equation
\begin{equation} \label{eqn 05.09.1}
\partial^{\alpha}_{t}u=a^{ij}u_{x^{i}x^{j}}+b^{i}u_{x^{i}}+cu +f, \quad t>0\,; \quad u(0,\cdot)=1_{\alpha>1}\partial_{t}u(0,\cdot)=0.
\end{equation} 
Moreover, the solution $u$ satisfies
\begin{equation} \label{eqn 05.08.1-1}
\|u\|_{\bH^{\alpha,\gamma+2}_{q,p}(w_{2},w_{1},T)}\leq N\|f\|_{\bH^{\gamma}_{q,p}(w_{2},w_{1},T)},
\end{equation}
where the constant $N$ depends only on $\alpha,d,p,q,\gamma,\delta,w_{1},w_{2}$, and $T$.
\end{theorem}

\begin{remark}
The presence of weights in estimates makes one to find other approaches to prove Theorem \ref{theorem 5.1} even for $\partial^{\alpha}_{t}u=\Delta u +f$ with $p=q=2$ and $\gamma=0$.

Indeed, by \cite{benedetto1992} for any $h\in \cS(\bR)$, we have
$$
\int_{\bR}|\mathcal{F}(h)(\xi)|^{2}w_{1}(1/\xi) d\xi\leq N \int_{\bR}|h(x)|^{2}w_{1}(x)dx
$$
if and only if $w\in A_{2}(\bR)$ is an even function on $\bR$ increasing on $(0,\infty)$. Hence, if we assume the Parseval's identity for $L_{2}(w_{1})$, then we have
$$
\int_{\bR}|\mathcal{F}(h)(\xi)|^{2}w_{1}(1/\xi) d\xi\leq N \int_{\bR}|\mathcal{F}(h)(\xi)|^{2}w_{1}(\xi)d\xi.
$$
If we take $w_{1}(\xi)=|\xi|^{1/2}\in A_{2}(\bR)$, then we have contradiction from the above inequality and the arbitrariness of $h$. Therefore, the Parseval's identity does not hold in weighted $L_{2}$ space, and thus one cannot simply get an $L_{2}$-theory via Fourier transform.
\end{remark}

\mysection{Interpolation and Uniform localization theorem of weighted Sobolev spaces}

First of all, we introduce Littlewood-Paley operator. Let $\cS=\cS(\R^{d})$ be a space of Schwartz functions defined on $\R^{d}$, $\cD$ be a space of tempered distribution, and let $\Psi$ be a radial function in $\cS$ whose Fourier transform $\hat{\Psi}(\xi)$ has support in the strip $\{\xi\in\R^{d}|\frac{1}{2}\leq |\xi|\leq 2\}$, $\hat{\Psi}(\xi)>0$ for $\frac{1}{2}<\xi<2$, and satisfying
\begin{equation}\label{eqn 09.17.14:14}
\sum_{j\in\bZ}\hat{\Psi}(2^{-j}\xi)=1 \quad \textrm{for} \quad \xi\neq0.
\end{equation}
Define $\Psi_{j}$ and $\Psi_{0}$ by
\begin{equation}\label{eqn 09.17.14:15}
\begin{aligned}
\hat{\Psi}_{j}(\xi)&=\hat{\Psi}(2^{-j}\xi) \quad j=\pm1,\pm2,\dots ,
\\
\hat{\Psi}_{0}(\xi)&=1-\sum_{j=1}^{\infty}\hat{\Psi}_{j}(\xi).
\end{aligned}
\end{equation}
The following lemma is a weighted version of the Littlewood-Paley inequality. For the proof, see e.g. \cite[Theorem 1]{kurtz1980littlewood}.
\begin{lemma}\label{lem 09.17.14:13}
Let $1<p<\infty$, $w_{1}\in A_{p}(\bR^{d})$. For any $f\in L_{p}(w)$, it holds that
\begin{equation}\label{eqn 09.17.14:16}
\begin{gathered}
\left\|\left(\sum_{j\in\bZ}|\Psi_{j}\ast f|^{2}\right)^{1/2}\right\|_{L_{p}(w_{1})}\leq N_{1} \|f\|_{L_{p}(w_{1})}\leq N_{2} \left\|\left(\sum_{j\in\bZ}|\Psi_{j}\ast f|^{2}\right)^{1/2}\right\|_{L_{p}(w_{1})},
\end{gathered}
\end{equation}
where the constant $N_{1}$ and $N_{2}$ depend only on $d,p,w_{1}$, and $\Psi$.
\end{lemma}

\begin{lemma}\label{lem 09.17.14:42}
Suppose that $\Phi\in\cS$, and let $\Phi_{j}(x)=2^{jd}\Phi(2^{j}x)$ for $j\in\bZ$. Define $K_{j}(x,y)=\Phi_{j}(x-y)$. Then it holds that
\begin{equation*}
\begin{gathered}
\left|\sum_{j\in\bZ}|K_{j}(x,y)|^{2}\right|^{1/2}\leq \frac{N}{|x-y|^{d}},
\\
\left|\sum_{j\in\bZ}|K_{j}(x,y)-K(x',y)|^{2}\right|^{1/2}\leq \frac{N|x-x'|}{(|x-y|+|x'-y|)^{d+1}},
\end{gathered}
\end{equation*}
provided that $|x-x'|\leq \frac{1}{2}\max{\{|x-y|,|x'-y|\}}$, where the constant $N$ depends only on $d$, and $\Phi$.
\end{lemma}
\begin{proof}
Take $x,y\in\R^{d}$. Let $j_{0}=j_{0}(x,y)$, the smallest integer so that $2^{j_{0}}|x-y|>1$. This implies that
\begin{equation}\label{eqn 03.20.19:45}
1<2^{j_{0}}|x-y|\leq2.
\end{equation}
Note that due to the choice of $j_{0}$, we have \eqref{eqn 03.20.19:45}. From this, we have
\begin{equation*}
\begin{aligned}
\left|\sum_{j\in\bZ}|K_{j}(x,y)|^{2}\right|^{1/2}&\leq \left|\sum_{j<j_{0}}|K_{j}(x,y)|^{2}\right|^{1/2}+\left|\sum_{j\geq j_{0}}|K_{j}(x,y)|^{2}\right|^{1/2}
\\
&\leq N(d,\Phi) \left|\sum_{j<j_{0}}2^{2jd}\right|^{1/2}+N(d,\Phi)  \left|\sum_{j\geq j_{0}}\left|2^{jd}\frac{1}{|2^{j}(x-y)|^{a}}\right|^{2}\right|^{1/2}
\\
&\leq N(d,\Phi)\frac{1}{|x-y|^{d}}+N(d,\Phi)2^{j_{0}(d-a)}\frac{1}{|x-y|^{a}}
\\
&\leq N(d,\Phi)\frac{1}{|x-y|^{d}},
\end{aligned}
\end{equation*}
where $a>d$, and the last inequality holds due to \eqref{eqn 03.20.19:45}. Note that since $\Phi\in\cS$, one can choose such $a$. 
\\
Now take $x,x',y\in\bR^{d}$ satisfying $|x-x'|\leq \frac{1}{2}\max{\{|x-y|,|x'-y|\}}$ and let $j_{1}$ be the smallest integer so that $2^{j_{1}}(|x-y|+|x'-y|)>1$. Like \eqref{eqn 03.20.19:45}, we get
\begin{equation}\label{eqn 03.30.20:03}
1<2^{j_{1}}(|x-y|+|x'-y|)\leq 2.
\end{equation}
Observe that by the mean value theorem,
\begin{equation*}
\begin{aligned}
&\left|\sum_{j\in\bZ}|K_{j}(x,y)-K_{j}(x',y)|^{2}\right|^{1/2} 
\\
&\leq \left|\sum_{j\in\bZ}\bigg||x-x'||\nabla\Phi_{j}(\theta(x,x',y))|\bigg|^{2}\right|^{1/2}
\\
&\leq  \left|\sum_{j<j_{1}}\bigg|2^{j(d+1)}|x-x'||\nabla\Phi(2^{j}\theta(x,x',y))|\bigg|^{2}\right|^{1/2}
\\
&\quad\quad+\left|\sum_{j\geq j_{1}}\bigg|2^{j(d+1)}|x-x'||\nabla\Phi(2^{j}\theta(x,x',y))|\bigg|^{2}\right|^{1/2}
\\
&\leq N(d,\Phi)|x-x'|2^{(d+1)j_{1}}
+ N(d,\Phi) \left|\sum_{j\geq j_{1}}\bigg|\frac{2^{j(d+1)}|x-x'|}{|2^{j}\theta(x,x',y)|^{d+b}}\bigg|^{2}\right|^{1/2}
\\
&\leq N(d,\Phi)|x-x'|2^{(d+1)j_{1}}+N(d,\Phi)\frac{|x-x'|2^{j_{1}(1-b)}}{|\theta(x,x',y)|^{d+b}},
\end{aligned}
\end{equation*}
where $b>1$ (recall that since $\Phi\in \cS$ we can choose such $b$), and $\theta(x,x',y)$ is a point in the line segment connecting $x-y$, and $x'-y$ . Note that there exist $\lambda,\lambda'\in(0,1)$ so that
\begin{equation*}
\begin{gathered}
\theta(x,x',y)=(x-y)+\lambda(x-x'),
\\
\theta(x,x',y)=(x'-y)+\lambda'(x-x').
\end{gathered}
\end{equation*}
Since $|x-x'|\leq \frac{1}{2}\max{\{|x-y|,|x'-y|\}}$, one can check that
\begin{equation}\label{eqn 01.07.16:50}
\frac{1}{4}(|x-y|+|x'-y|)\leq |\theta(x,x',y)| \leq (|x-y|+|x'-y|).
\end{equation}
Therefore, by \eqref{eqn 03.30.20:03} we get
\begin{equation*}
\left|\sum_{j\in\bZ}|K_{j}(x,y)-K_{j}(x',y)|^{2}\right|^{1/2} \leq N(d,\Phi) \frac{|x-x'|}{(|x-y|+|x'-y|)^{d+1}}.
\end{equation*}
The lemma is proved.
\end{proof}
Note that Lemma \ref{lem 09.17.14:42} means  $\{\Psi_{j}:j\in\bZ\}$ is a $l_{2}$-valued standard kernel. Therefore, the following lemma holds.
\begin{lemma}\label{lem 09.17.16:14}
Let $1<p<\infty$, $w_{1}\in A_{p}(\bR^{d})$, and let $f_{k}$ be a sequence of functions satisfying
\begin{equation*}
\left\|\big(\sum_{k\in\bZ}|f_{k}|^{2}\big)^{1/2}\right\|_{L_{p}(w_{1})}<\infty.
\end{equation*}
Take $K_{j}$ from Lemma \ref{lem 09.17.14:42}. Then it holds that
\begin{equation}\label{eqn 09.17.21:31}
\left\|\big(\sum_{j\in\bZ}\sum_{k\in\bZ}|K_{j}\ast f_{k}|^{2}\big)^{1/2}\right\|_{L_{p}(w_{1})}\leq N \left\|\big(\sum_{k\in\bZ}|f_{k}|^{2}\big)^{1/2}\right\|_{L_{p}(w_{1})},
\end{equation}
where the constant $N$ depends only on $d,p,\Phi$, and $w_{1}$. Moreover, for $\Psi_{j}$ it holds that
\begin{equation}\label{eqn 09.27.15:11}
\left\|\big(\sum_{k\in\bZ}|f_{k}|^{2}\big)^{1/2}\right\|_{L_{p}(w_{1})}\leq N \left\|\big(\sum_{j\in\bZ}\sum_{k\in\bZ}|\Psi_{j}\ast f_{k}|^{2}\big)^{1/2}\right\|_{L_{p}(w_{1})},
\end{equation}
\end{lemma}
\begin{proof}
From Lemma \ref{lem 09.17.14:42} for each fixed $x\in\bR^{d}$, and $a\in \bR$, the map 
$$
a\to  \{aK_{j}(x):j\in\bZ\}
$$ 
is a bounded linear from $\bR$ to $l_{2}$ satisfying $(D_{\infty})$ condition in \cite{de1986calderon}, and thus it satisfies $(D_{r})$ condition therein for any $1<r<\infty$. Also note that by Remark \ref{rmk 09.27.21:52} (ii) we can find $1<q<p$ such that $w\in A_{q}$. Therefore, applying \cite[Theorem 1.6]{de1986calderon} with $1<r<\infty$ such that $q<p/r'<p$, where $r'=r/r-1$, we have \eqref{eqn 09.17.21:31}. For \eqref{eqn 09.27.15:11}, see \cite[Theorem 3.1]{de1986calderon}.
\end{proof}

\begin{remark}\label{rmk 09.18.19:25}
(i) By taking $f_{k}=f1_{k=0}$ in Lemma \ref{lem 09.17.16:14} we can observe that the operator $T(f)=\{K_{j}\ast f\}_{j\in\bZ}$ maps $L_{p}(w_{1})$ to $L_{p}(l_{2},w_{1})$. 

(ii) Also observe that if $K$ is radial and has real-valued Fourier transform then we have
\begin{equation*}
\begin{gathered}
\int_{\R^{d}}(K_{j}\ast f)\bar{g}dx=\int_{\R^{d}}\hat{K}_{j}\hat{f}\bar{\hat{g}}d\xi=\int_{\R^{d}}\hat{f}\overline{\hat{K}_{j}\hat{g}}d\xi=\int_{\R^{d}}f\,\overline{(K_{j}\ast g)}dx,
\\
\int_{\R^{d}}(K_{j}\ast f)gdx=\int_{\R^{d}}f (K_{j}\ast g) dx.
\end{gathered}
\end{equation*}
By using these we can check that the adjoint operator $T^{\ast}$ of $T$ satisfies $T^{\ast}(g)=\sum_{j}K_{j}\ast g_{j}$ and it maps $L_{p'}(l_{2},\tilde{w}_{1})$ to $L_{p'}(\tilde{w}_{1})$, where $\tilde{w}_{1}=w_{1}^{-1/(p-1)}$. Since $1<p<\infty$ is arbitrary, for a sequence of functions $\{f_{j}:j\in\bZ\}\in L_{p}(l_{2},w_{1})$ we have
\begin{equation}\label{eqn 09.18.20:05}
\|\sum_{j\in\bZ}K_{j}\ast f_{j}\|_{L_{p}(w_{1})}\leq N \|\big(\sum_{j\in\bZ}|f_{j}|^{2}\big)^{1/2}\|_{L_{p}(w_{1})}.
\end{equation}
\end{remark}

The following theorem is a Fourier multiplier theorem for $L_{p}(w_{1})$, whose proof is containted in \cite{kurtz1980littlewood} (for more general version of theorem, see e.g. \cite{fackler2020}).
\begin{theorem}\label{thm 01.05.14:43}
Let $k>d$, $1<p<\infty$ and $w_{1}\in A_{p}(\bR^{d})$. Suppose that $T$ is a linear map on $\mathcal{S}(\bR^{d})$ defined as
$$
Tf(x)=\cF^{-1}(m \cF f)(x),
$$
where $m\in C^{k}(\bR^{d}\setminus \{0\})$ satisfies
$$
\sup_{r>0} r^{2|\beta|-d} \int_{r\leq |\xi|\leq 2r}\left| D^{\beta}m(\xi)\right|^{2} d\xi \leq B
$$
for all multi-index $\beta$ such that $|\beta|\leq k$, and for a positive constant $B$. Then we can extend $T$ to $L_{p}(w_{1})$, and  for any $f\in L_{p}(w_{1})$ it holds that
$$
\|Tf\|_{L_{p}(w_{1})} \leq N \|f\|_{L_{p}(w_{1})},
$$
where the constant $N$ depends only on $d,p,w_{1}$, and $B$.
\end{theorem}

By applying the above lemmas, we get the following Littlewood-Paley characterization for $H^{\gamma}_{p}(w_{1})$ whose proof is similar to that of \cite[Theorem 6.2.6]{grafakos2009modern}. The only difference is that we use the weighted inequalities. 
\begin{theorem}\label{thm 09.17.21:54}
Let $1<p<\infty$, $w_{1}\in A_{p}(\bR^{d})$. Then the following hold;

(i) If $u\in H^{\gamma}_{p}(w_{1})$, then there is a constant $N=N(d,p,\gamma,w_{1},\Psi)$ such that
\begin{equation}\label{eqn 09.17.21:55}
\begin{gathered}
\|\Psi_{0}\ast u\|_{L_{p}(w_{1})}+\|\big(\sum_{j=1}^{\infty}|2^{\gamma j}\Psi_{j}\ast u|^{2}\big)^{1/2}\|_{L_{p}(w_{1})}\leq N \|u\|_{H^{\gamma}_{p}(w_{1})}.
\end{gathered}
\end{equation}

(ii) Also if $u$ is a tempered distribution so that the left hand side of \eqref{eqn 09.17.21:55} is finite, then $u\in H^{\gamma}_{p}(w_{1})$, and there is a constant $N=N(d,p,\gamma,w_{1},\Psi)$ so that
\begin{equation}\label{eqn 09.18.10:40}
\begin{gathered}
\|u\|_{H^{\gamma}_{p}(w_{1})} \leq N\left(\|\Psi_{0}\ast u\|_{L_{p}(w_{1})}+\|\big(\sum_{j=1}^{\infty}|2^{\gamma j}\Psi_{j}\ast u|^{2}\big)^{1/2}\|_{L_{p}(w_{1})}\right)
\end{gathered}
\end{equation}
\end{theorem}
\begin{proof}
(i) Suppose that $u\in H^{\gamma}_{p}(w_{1})$. Note that for $j\geq 2$
\begin{equation*}
2^{\gamma j}\hat{\Psi}_{j}(\xi)\hat{u}(\xi)=|2^{-j}\xi|^{-\gamma}\hat{\Psi}_{j}(\xi)|\xi|^{\gamma}\hat{u}(\xi)=|2^{-j}\xi|^{-\gamma}\hat{\Psi}_{j}(\xi)|\xi|^{\gamma}(1-\hat{\Psi}_{0}(\xi))\hat{u}(\xi)
\end{equation*}
since the support of $\hat{\Psi}_{0}$ does not intersect the support of $\hat{\Psi}_{j}$. 
Therefore, by defining
\begin{equation*}
\tilde{u}=\mathcal{F}^{-1}\left\{|\cdot|^{\gamma}(1-\hat{\Psi}_{0})\hat{u}\right\},
\end{equation*}
we get
\begin{equation*}
\|\big(\sum_{j\geq2}|2^{\gamma j}\Psi_{j}\ast u|^{2}\big)^{1/2}\|_{L_{p}(w_{1})}=\|\big(\sum_{j\geq2}|\Psi_{\gamma,j}\ast \tilde{u}|^{2}\big)^{1/2}\|_{L_{p}(w_{1})},
\end{equation*}
where $\hat{\Psi}_{\gamma,j}(\xi)=|2^{-j}\xi|^{-\gamma}\hat{\Psi}_{j}(\xi)$. By Remark \ref{rmk 09.18.19:25} (i), we have
\begin{equation*}
\|\big(\sum_{j\geq2}|2^{\gamma j}\Psi_{j}\ast u|^{2}\big)^{1/2}\|_{L_{p}(w_{1})} \leq N \|\tilde{u}\|_{L_{p}(w_{1})}.
\end{equation*}
Note that
\begin{equation*}
\begin{aligned}
\tilde{u}&=\cF^{-1}(|\cdot|^{\gamma}(1-\hat{\Psi}_{0})\hat{u})
\\
&=\cF^{-1}\left(\frac{|\cdot|^{\gamma}
(1-\hat{\Psi}_{0})}{(1+|\cdot|^{2})^{\gamma/2}}(1+|\cdot|^{2})^{\gamma/2}\hat{u}\right).
\end{aligned}
\end{equation*}
Applying Theorem \ref{thm 01.05.14:43} to the multiplier
\begin{equation*}
m(\xi)=\frac{|\xi|^{\gamma}(1-\hat{\Psi}_{0})(\xi)}{(1+|\xi|^{2})^{\gamma/2}},
\end{equation*}
we have
\begin{equation*}
\|\big(\sum_{j\geq2}|2^{\gamma j}\Psi_{j}\ast u|^{2}\big)^{1/2}\|_{L_{p}(w_{1})}\leq N \|\tilde{u}\|_{L_{p}(w_{1})}\leq N \|u\|_{H^{\gamma}_{p}(w_{1})}.
\end{equation*}
Similarly for $j=0,1$,
\begin{equation*}
\begin{gathered}
2^{\gamma}\hat{\Psi}_{1}(\xi)\hat{u}(\xi)=2^{\gamma}\frac{\hat{\Psi}(\frac{1}{2}\xi)}{(1+|\xi|^{2})^{\gamma/2}}(1+|\xi|^{2})^{\gamma/2}\hat{u}(\xi),
\\
\hat{\Psi}_{0}(\xi)\hat{u}(\xi)=\frac{\hat{\Psi}_{0}(\xi)}{(1+|\xi|^{2})^{\gamma/2}}(1+|\xi|^{2})^{\gamma/2}\hat{u}(\xi),
\end{gathered}
\end{equation*}
and it holds that
\begin{equation*}
\begin{gathered}
\|2^{\gamma}\Psi_{1}\ast u\|_{L_{p}(w_{1})}\leq N \|u\|_{H^{\gamma}_{p}(w_{1})},
\\
\|\Psi_{0}\ast u\|_{L_{p}(w_{1})} \leq N \|u\|_{H^{\gamma}_{p}(w_{1})}
\end{gathered}
\end{equation*}
by Theorem \ref{thm 01.05.14:43}. Hence, we have \eqref{eqn 09.17.21:55}.

(ii) Now suppose that $u\in\cD$ satisfies
\begin{equation*}
\|\Psi_{0}\ast u\|_{L_{p}(w_{1})}+\|\big(\sum_{j=1}^{\infty}|2^{\gamma j}\Psi_{j}\ast u|^{2}\big)^{1/2}\|_{L_{p}(w_{1})}<\infty.
\end{equation*}
Note that
\begin{equation*}
(1+|\xi|^{2})^{\gamma/2}\hat{u}=\hat{\Psi}_{0}(\xi)(1+|\xi|^{2})^{\gamma/2}\hat{u}+(1-\hat{\Psi}_{0}(\xi))(1+|\xi|^{2})^{\gamma/2}\hat{u}=:\hat{u}_{0}+\hat{u}_{1}.
\end{equation*}
Take $\Pi_{0}\in\cS$, whose Fourier transform $\hat{\Pi}_{0}$ has compact support and equals 1 on the support of $\hat{\Psi}_{0}$. Then we have
\begin{equation*}
\hat{u}_{0}(\xi)=\hat{\Pi}_{0}(\xi)(1+|\xi|^{2})^{\gamma/2}  \hat{\Psi}_{0}(\xi)\hat{u}(\xi).
\end{equation*}
Since $\hat{\Pi}_{0}(1+|\xi|^{2})^{\gamma/2}$ satisfies the condition in Theorem \ref{thm 01.05.14:43}, we have
\begin{equation*}
\|u_{0}\|_{L_{p}(w_{1})}\leq N \|\Psi_{0}\ast u\|_{L_{p}(w_{1})}.
\end{equation*}
Let $\Pi_{1}$ be a smooth function so that $\hat{\Pi}_{1}=0$ in the neighborhood of the origin and equals 1 on the support of $(1-\hat{\Psi}_{0})$. Observe that
\begin{equation*}
\hat{u}_{1}(\xi)=|\xi|^{-\gamma}(1+|\xi|^{2})^{\gamma/2}\hat{\Pi}_{1}(\xi)(1-\hat{\Psi}_{0}(\xi))|\xi|^{\gamma}\hat{u}(\xi).
\end{equation*}
Since $|\xi|^{-\gamma}(1+|\xi|^{2})^{\gamma/2}\hat{\Pi}_{1}(\xi)$ satisfies the condition in Theorem \ref{thm 01.05.14:43}, we have
\begin{equation*}
\|u_{1}\|_{L_{p}(w_{1})}\leq N \|\tilde{u}_{1}\|_{L_{p}(w_{1})},
\end{equation*}
where
\begin{equation*}
\begin{aligned}
\tilde{u}_{1}&=\cF^{-1}(|\cdot|^{\gamma}(1-\hat{\Psi}_{0})\hat{u}).
\end{aligned}
\end{equation*}
Let  $\Pi_{2}\in\cS$ such that $\hat{\Pi}_{2}=1$ on the support of $\hat{\Psi}$ and $\hat{\Pi}_{2}=0$ near the origin. By using the property of $\hat{\Psi}_{0}$,
\begin{equation*}
\mathcal{F}\tilde{u}_{1}=\sum_{j=1}^{\infty}|\xi|^{\gamma}\hat{\Psi}_{j}(\xi)\hat{\Pi}_{2}(2^{-j}\xi)\hat{u}(\xi)=\sum_{j=1}^{\infty}2^{\gamma j}\hat{\Psi}_{j}(\xi)\hat{\Theta}_{j}(\xi)\hat{u}(\xi),
\end{equation*}
where $\hat{\Theta}(\xi)=|\xi|^{\gamma}\hat{\Pi}_{2}(\xi)$, and $\hat{\Theta}_{j}(\xi)=\hat{\Theta}(2^{-j}\xi)$. Thus it follows that
\begin{equation*}
\tilde{u}_{1}=\sum_{j=1}^{\infty}\Theta_{j}\ast ( 2^{\gamma j}\Psi_{j}\ast u)
\end{equation*}
By \eqref{eqn 09.18.20:05}  we have
\begin{equation*}
\|\tilde{u}_{1}\|_{L_{p}(w_{1})}\leq N \|\big(\sum_{j=1}^{\infty}|2^{\gamma j}\Psi_{j}\ast u|^{2}\big)^{1/2}\|_{L_{p}(w_{1})}.
\end{equation*}
So we obtain \eqref{eqn 09.18.10:40}, and the theorem is proved.
\end{proof}
Now we give an interpolation theorem of weighted Sobolev spaces.
\begin{theorem}\label{thm 09.24.13:17}
Let $1<p_{0},p_{1}<\infty$, $w^{0}_{1}\in A_{p_{0}}(\bR^{d})$,$w^{1}_{1}\in A_{p_{1}}(\bR^{d})$, $\gamma_{0},\gamma_{1}\in \bR$ and let $\theta\in(0,1)$. If we define 
\begin{gather*}
\gamma=(1-\theta)\gamma_{0}+\theta\gamma_{1}, \quad w_{1}=(w^{0}_{1})^{p(1-\theta)/p_{0}}(w^{1}_{1})^{p\theta/p_{1}},
\\
\frac{1}{p}=\frac{1-\theta}{p_{0}}+\frac{\theta}{p_{1}},
\end{gather*}
then we have
\begin{equation}\label{eqn 01.06.14:29}
[H^{\gamma_{0}}_{p_{0}}(w^{0}_{1}),H^{\gamma_{1}}_{p_{1}}(w^{1}_{1})]_{[\theta]}=H^{\gamma}_{p}(w_{1}),
\end{equation}
where the space $[H^{\gamma_{0}}_{p}(w^{0}_{1}),H^{\gamma_{1}}_{p}(w^{1}_{1})]_{[\theta]}$ is the complex interpolation space.
\end{theorem}
\begin{proof}
For $\mu\in \bR$, define $l^{\mu}_{2}$ as follows;
\begin{equation*}
l^{\mu}_{2}:=\{a=(a_{1},a_{2},\dots):\|a\|_{l^{\mu}_{2}}:=(\sum_{j=0}^{\infty}2^{\mu j}|a_{j}|^{2})^{1/2}<\infty\}.
\end{equation*}
By Theorem \ref{thm 09.17.21:54} (i) if we define $(\mathcal{I}u)_{j}=\Psi_{j}\ast u$ ($j=0,1,2,\dots$), then for any $1<q<\infty$, $w\in A_{q}(\bR^{d})$, and $u\in H^{\mu}_{q}(w)$, we have $\mathcal{I}u\in L_{q}(l^{\mu}_{2};wdx)$. This implies that
$$
\mathcal{I}[H^{\gamma_{0}}_{p_{0}}(w^{0}_{1})+H^{\gamma_{1}}_{p_{1}}(w^{1}_{1})]\subset [L_{p_{0}}(l^{\gamma_{0}}_{2};w^{0}_{1}dx)+L_{p_{1}}(l^{\gamma_{1}}_{2};w^{1}_{1}dx)].
$$
Now  define a map $\mathcal{P}$ on $L_{p_{0}}(l_{2}^{\gamma_{0}};w^{0}_{1}dx)+L_{p_{1}}(l_{2}^{\gamma_{1}};w^{1}_{1}dx)$ by 
$$
\mathcal{P}u=(\Psi_{0}+\Psi_{1})\ast u_{0}+\sum_{j=1}^{\infty}(\Psi_{j-1}+\Psi_{j}+\Psi_{j+1})\ast u_{j}.
$$
Since
\begin{equation}\label{eqn 02.09.14:34}
\begin{gathered}
\Psi_{0}=\Psi_{0}\ast(\Psi_{0}+\Psi_{1}),
\\
\Psi_{j}=\Psi_{j}\ast(\Psi_{j-1}+\Psi_{j}+\Psi_{j+1}),\quad j=1,2,\dots,
\\
\Psi_{i}\ast\Psi_{j}=0 \quad \forall\, |i-j|\geq2,
\end{gathered}
\end{equation}
it follows that
\begin{equation*}
\begin{aligned}
\Psi_{0}\ast\mathcal{P}u&=\Psi_{0}\ast u_{0}+ (\Psi_{0}+\Psi_{1})\ast\Psi_{0}\ast u_{1},
\\
\Psi_{j}\ast\mathcal{P}u&=(\Psi_{j-1}+\Psi_{j})\ast \Psi_{j}\ast u_{j-1}
\\
&\quad + \Psi_{j}\ast u_{j}+ (\Psi_{j}+\Psi_{j+1})\ast\Psi_{j}\ast u_{j+1} \quad j=1,2,\dots.
\end{aligned}
\end{equation*}
Therefore, we have for $i=0,1$
\begin{equation*}
\begin{aligned}
&\|(\sum_{j=0}^{\infty}2^{\gamma_{i} j}\Psi_{j}\ast\mathcal{P}u|^{2})^{1/2}\|_{L_{p_{i}}(w^{i}_{1})}
\\
&\leq \||\sum_{j=0}^{\infty}|\Psi_{j}\ast 2^{\gamma_{i} j}u_{j}|^{2}|^{1/2}\|_{L_{p_{i}}(w^{i}_{1})}
\\
&\quad+2^{-\gamma_{i}} \|(\sum_{j=0}^{\infty}|(\Psi_{j+1}+\Psi_{j})\ast\Psi_{j}\ast 2^{\gamma_{i} (j+1)}u_{j+1}|^{2})^{1/2}\|_{L_{p_{i}}(w^{i}_{1})}
\\
&\quad+2^{\gamma_{i}} \|(\sum_{j=1}^{\infty}|(\Psi_{j-1}+\Psi_{j})\ast\Psi_{j}\ast 2^{\gamma_{i} (j-1)}u_{j-1})^{2}|^{1/2}\|_{L_{p_{i}}(w^{i}_{1})}.
\end{aligned}
\end{equation*}
If we apply \eqref{eqn 09.18.20:05}, then it follows that
$$
\|(\sum_{j=0}^{\infty}2^{\gamma_{i} j}\Psi_{j}\ast\mathcal{P}u|^{2})^{1/2}\|_{L_{p_{i}}(w^{i}_{1})} \leq N \|u\|_{L_{p_{i}}(l^{\gamma_{i}}_{2},w^{i}_{1})} \quad i=0,1,
$$
and this gives $\mathcal{P}u\in [H^{\gamma_{0}}_{p_{0}}(w^{0}_{1})+H^{\gamma_{1}}_{p_{1}}(w^{1}_{1})]$ by Theorem \ref{thm 09.17.21:54} (ii). Also, one can check that $\mathcal{P}\mathcal{I}$ is the indentity map defined on $H^{\mu}_{q}(w)$, and on  $[H^{\gamma_{0}}_{p_{0}}(w^{0}_{1})+H^{\gamma_{1}}_{p_{1}}(w^{1}_{1})]$ by using \eqref{eqn 02.09.14:34}.  

On the other hand, by \cite[Theorem 1.18.5]{triebel1977}, and \cite[Theorem 5.6.3]{bergh2012interpolation} we have
\begin{equation*}
\begin{gathered}
\quad[L_{p_{0}}(A_{0};w^{0}_{1}dx),L_{p_{1}}(A_{1};w^{1}_{1}dx)]_{[\theta]}=L_{p}([A_{0},A_{1}]_{[\theta]};w_{1}dx),
\\
[l^{\gamma_{0}}_{2},l^{\gamma_{1}}_{2}]_{[\theta]}=l^{\gamma}_{2},
\end{gathered}
\end{equation*}
where $A_{0}$, and $A_{1}$ are arbitrary Banach spaces. Hence, it follows that 
\begin{equation*}
[L_{p_{0}}(l^{\gamma_{0}}_{2};w^{0}_{1}dx),L_{p_{1}}(l^{\gamma_{1}}_{2};w^{1}_{1}dx)]_{[\theta]}=L_{p}(l^{\gamma}_{2},w_{1}).
\end{equation*}
Also by \cite[Exercise 3.13.18]{bergh2012interpolation}, we have
\begin{equation*}
\begin{gathered}
\mathcal{I}:[H^{\gamma_{0}}_{p_{0}}(w^{0}_{1}),H^{\gamma_{1}}_{p_{1}}(w^{1}_{1})]_{[\theta]}\to[L_{p_{0}}(l^{\gamma_{0}}_{2};w^{0}_{1}dx),L_{p_{1}}(l^{\gamma_{1}}_{2};w^{1}_{1}dx)]_{[\theta]},
\\
\mathcal{P}:[L_{p_{0}}(l^{\gamma_{0}}_{2};w^{0}_{1}dx),L_{p_{1}}(l^{\gamma_{1}}_{2};w^{1}_{1}dx)]_{[\theta]}\to[H^{\gamma_{0}}_{p}(w_{1}),H^{\gamma_{1}}_{p}(w_{1})]_{[\theta]},
\\
\mathcal{P}\mathcal{I}[H^{\gamma_{0}}_{p_{0}}(w^{0}_{1}),H^{\gamma_{1}}_{p_{1}}(w^{1}_{1})]_{[\theta]}=[H^{\gamma_{0}}_{p_{0}}(w^{0}_{1}),H^{\gamma_{1}}_{p_{1}}(w^{1}_{1})]_{[\theta]}.
\end{gathered}
\end{equation*}
Therefore, we have
\begin{equation*}
\begin{aligned}
H^{\gamma}_{p}(w_{1}) &=\mathcal{P}\mathcal{I}H^{\gamma}_{p}(w_{1})
\\
&\subset \mathcal{P}L_{p}(l^{\gamma}_{2};w_{1}dx)=\mathcal{P}[L_{p_{0}}(l^{\gamma_{0}}_{2};w^{0}_{1}dx),L_{p_{1}}(l^{\gamma_{1}}_{2};w^{1}_{1}dx)]_{[\theta]}
\\
&\subset [H^{\gamma_{0}}_{p_{0}}(w^{0}_{1}),H^{\gamma_{1}}_{p_{1}}(w^{1}_{1})]_{[\theta]},
\end{aligned}
\end{equation*}
and
\begin{equation*}
\begin{aligned}
{[H^{\gamma_{0}}_{p_{0}}(w^{0}_{1}),H^{\gamma_{1}}_{p_{1}}(w^{1}_{1})]}_{[\theta]}&=\mathcal{P}\mathcal{I}[H^{\gamma_{0}}_{p_{0}}(w^{0}_{1}),H^{\gamma_{1}}_{p_{1}}(w^{1}_{1})]_{[\theta]}
\\
&\subset \mathcal{P}[L_{p_{0}}(l^{\gamma_{0}}_{2};w^{0}_{1}dx),L_{p_{1}}(l^{\gamma_{1}}_{2};w^{1}_{1}dx)]_{[\theta]}=\mathcal{P}L_{p}(l^{\gamma}_{2};w_{1}dx)
\\
&\subset H^{\gamma}_{p}(w_{1}).
\end{aligned}
\end{equation*}
Therefore, we have \eqref{eqn 01.06.14:29}, and the theorem is proved.
\end{proof}

\begin{remark}\label{rmk 01.14.15:40}

(i) Note that by H\"older's inequality, we have 
$$
[w_{1}]_{p}\leq [w^{0}_{1}]^{p(1-\theta)/p_{0}}_{p_{0}}[w^{1}_{1}]^{p\theta/p_{1}}_{p_{1}},
$$
and thus $w_{1}$ is also in $A_{p}(\bR^{d})$.

(ii) From Theorem \ref{thm 09.24.13:17}, and the definition of complex interpolation space,
$$
\|u\|_{H^{\gamma}_{p}(w_{1})} = \|u\|_{[H^{\gamma_{0}}_{p_{0}}(w^{0}_{1}),H^{\gamma_{1}}_{p_{1}}(w^{1}_{1})]_{[\theta]}} \leq \|u\|^{1-\theta}_{H^{\gamma_{0}}_{p_{0}}(w^{0}_{1})}\|u\|^{\theta}_{H^{\gamma_{1}}_{p_{1}}(w^{1}_{1})}.
$$
Therefore, for any $\varepsilon>0$, we have
\begin{equation}\label{eqn 01.14.16:41}
\|u\|_{H^{\gamma}_{p}(w_{1})} \leq \varepsilon\|u\|_{H^{\gamma_{0}}_{p_{0}}(w^{0}_{1})} + N(\varepsilon,\gamma_{0},\gamma_{1},\theta)\|u\|_{H^{\gamma_{1}}_{p_{1}}(w^{1}_{1})}.
\end{equation}
In particular, if $w_{1}=w^{0}_{1}=w^{1}_{1}=1$ this is standard interpolation inequality (see e.g. \cite{triebel1977,bergh2012interpolation}).
\end{remark}

\begin{lemma}\label{lem 01.06.16:05}
Let $1<p<\infty$, $\gamma\in\bR$, $w_{1}\in A_{p}(\mathbb{R}^{d})$, and let $a\in B^{|\gamma|}$. Then for any $u\in H^{\gamma}_{p}(w_{1})$, we have $au\in H^{\gamma}_{p}(w_{1})$, and
\begin{equation}\label{eqn 12.30.15:27}
\|au\|_{H^{\gamma}_{p}(w_{1})} \leq N \|a\|_{B^{|\gamma|}} \|u\|_{H^{\gamma}_{p}(w_{1})},
\end{equation}
where the constant $N$ depends only on $d,p,\gamma$, and  $|a|_{B^{|\gamma|}}$.
\end{lemma}
\begin{proof}
If $\gamma=0,1,2,\dots$, then by Remark \ref{rmk 05.08.1} (i) and product rule we have
\begin{equation}\label{eqn 12.30.16:19}
\|au\|_{H^{\gamma}_{p}(w_{1})} \leq N(d,p,\gamma) |a|_{B^{|\gamma|}}\|u\|_{H^{\gamma}_{p}(w_{1})}.
\end{equation}

Now assume that $\gamma=-1,-2,\dots$. Take $p'$ and $\tilde{w}_{1}$ from \eqref{dual} with $w_{1}$ in place of $w$. Then for any $v\in H^{-\gamma}_{p'}(\tilde{w}_{1})$ with $\|v\|_{H^{-\gamma}_{p'}(\tilde{w}_{1})}\leq 1$ we have
\begin{equation*}
\begin{aligned}
\left| (au,v) \right| &= \left| (u,av) \right| =  \left|\int_{\bR^{d}} (1-\Delta)^{\gamma/2}u(x) (1-\Delta)^{-\gamma/2}(av)(x) dx \right|
\\
&\leq \|u\|_{H^{\gamma}_{p}(w_{1})} \|av\|_{H^{-\gamma}_{p'}(\tilde{w}_{1})} \leq N |a|_{B^{|\gamma|}} \|u\|_{H^{\gamma}_{p}(w_{1})},
\end{aligned}
\end{equation*}
where the constant $N$ depends only on $d,p,\gamma$ Note that since $\Ccinf (\bR^{d})$ is dense in $H^{-\gamma}_{p'}(\tilde{w}_{1})$ one can define $(u,av)$ above (recall \eqref{eqn 5.12.2}). Therefore, we have 
\begin{equation}\label{eqn 12.30.16:21}
\|au\|_{H^{\gamma}_{p}(w_{1})} \leq N(d,p,\gamma) |a|_{B^{|\gamma|}}\|u\|_{H^{\gamma}_{p}(w_{1})}.
\end{equation}

Finally, assume that $\gamma$ is not an integer. Take $\kappa'$ from \eqref{eqn 01.06.14:42}, and let $\gamma=k+\nu$, where $k\in \bZ$ and $\nu\in(0,1)$. If we define a bilinear map 
$$
T : ( \mathcal{C}^{k+1+\kappa'}+\mathcal{C}^{k+\kappa'} )\times ( H^{k+1}_{p}(w_{1}) + H^{k}_{p}(w_{1}) ) \to ( H^{k+1}_{p}(w_{1}) +H^{k}_{p}(w_{1}) )
$$ 
by $T(a,u)=au$, then by \eqref{eqn 12.30.16:19}, \eqref{eqn 12.30.16:21}, and Remark \ref{rmk 01.06.14:48}, we have that for $a_{0}\in \mathcal{C}^{k+\kappa'},a_{1}\in \mathcal{C}^{k+1+\kappa'}$, and $u_{0}\in H^{k}_{p}(w_{1}),u_{1}\in H^{k+1}_{p}(w_{1})$, 
\begin{equation*}
\begin{gathered}
\|T(a_{0},u_{0})\|_{H^{k}_{p}(w_{1})} \leq N |a_{0}|_{\mathcal{C}^{k+\kappa'}} \|u_{0}\|_{H^{k}_{p}(w_{1})}
\\
\|T(a_{1},u_{1})\|_{H^{k+1}_{p}(w_{1})} \leq N |a_{1}|_{\mathcal{C}^{k+1+\kappa'}} \|u_{0}\|_{H^{k+1}_{p}(w_{1})}
\end{gathered}
\end{equation*}
by \eqref{eqn 12.30.16:19} and \eqref{eqn 12.30.16:21}. Therefore, by \cite[Theorem 4.4.1]{bergh2012interpolation} and \eqref{eqn 01.06.14:29} with 
\begin{gather*}
\theta=\nu \quad p_{0}=p_{1}=p, \quad w^{0}_{1}=w^{1}_{1}=w_{1},
\\
\gamma_{0}=k,\quad \gamma_{1}=k+1,
\end{gather*}
we have
$$
\|au\|_{H^{\gamma}_{p}(w_{1})} \leq N |a|_{[\mathcal{C}^{k+\kappa'},\mathcal{C}^{k+1+\kappa'}]_{[\nu]}}\|u\|_{H^{\gamma}_{p}(w_{1})}.
$$
By Remark \ref{rmk 01.06.14:48} and \cite[Theorem 6.4.5]{bergh2012interpolation} we have $[\mathcal{C}^{k+\kappa'},\mathcal{C}^{k+1+\kappa'}]_{[\nu]}=\mathcal{C}^{\gamma+\kappa'}$. Hence, again by Remark \ref{rmk 01.06.14:48} we have
$$
\|au\|_{H^{\gamma}_{p}(w_{1})} \leq N |a|_{\mathcal{C}^{\gamma+\kappa'}}\|u\|_{H^{\gamma}_{p}(w_{1})} \leq  N |a|_{B^{|\gamma|}}\|u\|_{H^{\gamma}_{p}(w_{1})}
$$
for any $a\in B^{|\gamma|}$, and $u\in H^{\gamma}_{p}(w_{1})$. The lemma is proved.
\end{proof}

\begin{remark}
Lemma \ref{lem 01.06.16:05} is called the pointwise multiplier theorem (see \cite[Section 2.8]{triebel1983} for detailed definition and pointwise multiplier theorem without weight).  If $w_{1}=1$, then Lemma \ref{lem 01.06.16:05} is a consequence of \cite[Lemma 5.2 (i)]{kry99analytic}. Note that from the proof, the additional H\"older regularity $\kappa'$ can not be easily removed when $\gamma$ is not an integer.
\end{remark}

\vspace{3mm}

Now we investigate uniform localization in $H^{\gamma}_{p}(w_{1})$. The argument comes from \cite{kry94Littlewood} which treats the case $w_{1}=1$.

For $\lambda>0$, define $L_{\lambda}=(\lambda-\Delta)$, and let $G_{m}(x)$ be the Green's function of the operator $(1-\Delta)^{m}$ for $m>0$. For multi-indices $\beta$ such that $|\beta|\leq 2m-1$, define
\begin{equation}\label{eqn 01.07.15:43}
G_{m,\lambda,\beta}(x):=\lambda^{(d-|\beta|)/2}D^{\beta}(G_{m}(\lambda^{1/2}x))=\lambda^{d/2}D^{\beta}G_{m}(\lambda^{1/2}x).
\end{equation}
Since $|\beta|\leq 2m-1$, $G_{m,\lambda,\beta}$ is integrable (see \cite[Section 12.7]{krylov2008lectures}) and $\|G_{m,\lambda,\beta}\|_{L_{1}}$ is independent of $\lambda$. Also by change of variables, for any $\phi\in \mathcal{S}(\bR^{d})$,
\begin{equation*}
\begin{aligned}
&\int_{\bR^{d}}(\lambda+|\xi|^{2})^{-m}e^{i\xi\cdot x} \mathcal{F}\phi(\xi)d\xi
\\
&=\lambda^{-m}\int_{\bR^{d}}(1+|\xi|^{2})^{-m}e^{i\xi\cdot (\lambda^{1/2} x)}\mathcal{F}(\phi(\lambda^{-1/2}\cdot))(\xi)d\xi.
\end{aligned}
\end{equation*}
This implies that the Green's function $G_{m,\lambda}$ of $(\lambda-\Delta)^{m}$ satisfies
\begin{equation*}
\begin{aligned}
\int_{\bR^{d}}G_{m,\lambda}(x-y)\phi(y)dy&=\lambda^{-m}\int_{\bR^{d}}G_{m}(\lambda^{1/2}x-y)\phi(\lambda^{-1/2}y)dy
\\
&=\lambda^{-m+d/2}\int_{\bR^{d}}G_{m}(\lambda^{1/2}x-\lambda^{1/2}y)\phi(y)dy.
\end{aligned}
\end{equation*}
From this, we have
\begin{gather}
\nonumber
G_{m,\lambda}(x)=\lambda^{-m+d/2}G_{m}(\lambda^{1/2}x) \quad \forall\, x\in \bR^{d},
\\
\label{eqn 01.25.15:34}
G_{m,\lambda,\beta}(x)=\lambda^{m-|\beta|/2}D^{\beta}G_{m,\lambda}(x) \quad \forall x\in\bR^{d},
\\
\label{eqn 01.06.20:43}
L^{m}_{\lambda}\phi(x)=\lambda^{m}[(1-\Delta)^{m}\phi(\lambda^{-1/2}\cdot)](\lambda^{1/2}x) \quad \forall\, x\in \bR^{d},
\end{gather}
and
\begin{gather}
\label{eqn 01.08.10:50}
\|L^{m}_{\lambda}\phi\|_{L_{p}(w_{1})} = \lambda^{m-\frac{d}{2p}}\|(1-\Delta)^{m}\phi_{\lambda}\|_{L_{p}(w^{\lambda}_{1})} \quad \forall \, m\in \bR,
\end{gather}
where $\phi_{\lambda}(x)=\phi(\lambda^{-1/2}x)$, and $w^{\lambda}_{1}(x)=w_{1}(\lambda^{-1/2}x)$.
\begin{lemma}
Let $1<p<\infty$, $w_{1}\in A_{p}(\bR^{d})$, $m=0,1,2,\dots$, $\lambda\geq1$, and $\varepsilon\in(0,1/2)$. Suppose that $\eta_{k}\in C^{\infty}(\bR^{d})$ for $k=1,2,\dots$, and assume that for any multi-index $\sigma$
$$
\sup_{x\in\bR^{d}}\sum_{k=1}^{\infty}|D^{\sigma}\eta_{k}(x)|\leq M(\sigma).
$$

(i) Then there exist constants $C^{m}_{\beta\sigma}$ such that for any $f\in \cS$
\begin{gather}
\label{eqn 01.06.20:02}
L^{m}_{\lambda}(\eta_{k}f)=\eta_{k}L^{m}_{\lambda}f+\sum_{0<|\sigma|, |\sigma|+|\beta|\leq 2m} C^{m}_{\beta\sigma}\lambda^{-|\sigma|/2}(D^{\sigma}\eta_{k})(G_{m,\lambda,\beta}\ast L^{m}_{\lambda}f),
\\
\label{eqn 01.06.20:02-2}
L^{m}_{\lambda}(\eta_{k}L^{-m}_{\lambda}f)=\eta_{k}f+\sum_{0<|\sigma|, |\sigma|+|\beta|\leq 2m} C^{m}_{\beta\sigma}\lambda^{-|\sigma|/2}(D^{\sigma}\eta_{k})(G_{m,\lambda,\beta}\ast f),
\\
\label{eqn 01.06.20:02-3}
L^{-m}_{\lambda}(\eta_{k}L^{m}_{\lambda}f)=\eta_{k}f+\sum_{0<|\sigma|, |\sigma|+|\beta|\leq 2m} C^{m}_{\beta\sigma}\lambda^{-|\sigma|/2}G_{m,\lambda,\beta}\ast( (D^{\sigma}\eta_{k})f),
\\
\label{eqn 01.06.20:02-4}
L^{-m}_{\lambda}(\eta_{k}f)=\eta_{k}L^{-m}_{\lambda}f+\sum_{0<|\sigma|, |\sigma|+|\beta|\leq 2m} C^{m}_{\beta\sigma}\lambda^{-|\sigma|/2}G_{m,\lambda,\beta}\ast( (D^{\sigma}\eta_{k})L^{-m}_{\lambda}f),
\end{gather}

(ii) For any $f\in \Ccinf(\bR^{d})$ we have
\begin{equation}\label{eqn 01.07.11:13}
\sum_{k=1}^{\infty}\|L^{\pm m/2}_{\lambda}\{L^{\varepsilon}_{\lambda}[\eta_{k}f]-\eta_{k}L^{\varepsilon}_{\lambda}f\}\|^{p}_{L_{p}(w_{1})} \leq N \lambda^{p(\varepsilon-1/2)}\|L^{\pm m/2}_{\lambda}f\|^{p}_{L_{p}(w_{1})},
\end{equation} 
where the constant $N$ depends only on $d,p,m,\varepsilon,M$ and $[w_{1}]_{p}$ 
\end{lemma}

\begin{proof}
(i) See \cite[Lemma 2.1 (i)]{kry94Littlewood}

(ii) \textit{Case 1.} $m=0$.
\\
 From the proof of \cite[Lemma 2.1]{kry94Littlewood}, we have
\begin{equation*}
\begin{aligned}
&L^{\varepsilon}_{\lambda}[\eta_{k}f]-\eta_{k}L^{\varepsilon}_{\lambda}f(x)
\\
&\quad=N(d,\varepsilon)\int_{0}^{\infty}e^{-\lambda t/2}\int_{\bR^{d}}t^{-d/2}\psi\left(\frac{y}{t^{1/2}} \right)[\eta_{k}(x-y)-\eta_{k}(x)]f(x-y)dy\frac{dt}{t^{1+\varepsilon}},
\end{aligned}
\end{equation*}
where $\psi(x)=(2\pi)^{-d/2}\exp (-|x|^{2}/2)$. Therefore, we have
\begin{equation*}
\begin{aligned}
&\sum_{k=1}^{\infty}|L^{\varepsilon}_{\lambda}[\eta_{k}f]-\eta_{k}L^{\varepsilon}_{\lambda}f(x)|^{p}
\\
&\quad\leq  N(d,\varepsilon) \sum_{k=1}^{\infty} \left| \int_{0}^{\infty}e^{-\frac{\lambda t}{2}}\int_{\bR^{d}}\left|t^{-\frac{d}{2}}\psi \left(\frac{y}{t^{1/2}} \right)[\eta_{k}(x-y)-\eta_{k}(x)]f(x-y)\right|dy  \frac{dt}{t^{1+\varepsilon}}\right|^{p}
\\
&\quad \leq N(d,\varepsilon,p,M)\left| \int_{0}^{\infty}e^{-\lambda t/2}t^{-1/2-\varepsilon} \int_{\bR^{d}}\left|t^{-d/2}\phi\left(\frac{y}{t^{1/2}}\right)f(x-y)\right|dy dt\right|^{p},
\end{aligned}
\end{equation*}
where $\phi(x)=\psi(x/2)$. Since $\varepsilon\in(0,1/2)$, the function $e^{-\lambda t/2}t^{-1/2-\varepsilon}$ is integrable on $(0,\infty)$ in $t$ and the integral equals $N(\varepsilon)\lambda^{\varepsilon-1/2}$. Therefore, by Jensen's inequality we have
\begin{equation*}
\begin{aligned}
&\sum_{k=1}^{\infty}\|L^{\varepsilon}_{\lambda}[\eta_{k}f]-\eta_{k}L^{\varepsilon}_{\lambda}f\|^{p}_{L_{p}(w_{1})}
\\
&\leq N(d,\varepsilon,M)\lambda^{(\varepsilon-1/2)(p-1)}  \int_{0}^{\infty}e^{-\lambda t/2}\int_{\bR^{d}}\left|\phi_{t}\ast f (x)\right|^{p}w_{1}(x)dx\frac{dt}{t^{1/2+\varepsilon}},
\end{aligned}
\end{equation*}
where $\phi_{t}(x)=t^{-d/2}\phi(t^{-1/2}x)$. Since $\phi\in \cS(\bR^{d})$, by Remark \ref{rmk 05.08.1} (ii) we have
\begin{equation*}
\begin{aligned}
\sum_{k=1}^{\infty}\|L^{\varepsilon}_{\lambda}[\eta_{k}f]-\eta_{k}L^{\varepsilon}_{\lambda}f\|^{p}_{L_{p}(w_{1})} &\leq N \lambda^{(\varepsilon-1/2)(p-1)} \int_{0}^{\infty} e^{-\lambda t/2}\frac{dt}{t^{1/2+\varepsilon}}\|f\|^{p}_{L_{p}(w_{1})}
\\
& \leq N \lambda^{(\varepsilon-1/2)p}\|f\|^{p}_{L_{p}(w_{1})},
\end{aligned}
\end{equation*}
where the constant $N$ depends only on $d,p,\varepsilon,M$, and $[w_{1}]_{p}$. This gives \eqref{eqn 01.07.11:13} for $m=0$. 

\textit{Case 2.} $m$ is an even integer. 
\\
By \eqref{eqn 01.06.20:02-3} we have
\begin{equation}\label{eqn 01.07.16:58}
\begin{aligned}
&L^{-m/2}_{\lambda}\{L^{\varepsilon}_{\lambda}[\eta_{k}f]-\eta_{k}L^{\varepsilon}_{\lambda}f\}
\\
&\quad = L^{\varepsilon}_{\lambda}(L^{-m/2}_{\lambda}[\eta_{k}L^{m/2}_{\lambda}g])-L^{-m/2}_{\lambda}[\eta_{k}L^{m/2}_{\lambda}(L^{\varepsilon}_{\lambda}g)]
\\
&\quad = L^{\varepsilon}_{\lambda}[\eta_{k}g]-\eta_{k}L^{\varepsilon}_{\lambda}g
\\
&\quad \quad+ \sum_{0<|\sigma|, |\sigma|+|\beta|\leq m} C^{m/2}_{\beta\sigma}\lambda^{-|\sigma|/2}G_{\frac{m}{2},\lambda,\beta}\ast\{L^{\varepsilon}_{\lambda}[(D^{\sigma}\eta_{k})g]-(D^{\sigma}\eta_{k}L^{\varepsilon}_{\lambda}g)\}
\\
&\quad = L^{\varepsilon}_{\lambda}[\eta_{k}g]-\eta_{k}L^{\varepsilon}_{\lambda}g
\\
&\quad \quad+ \sum_{0<|\sigma|, |\sigma|+|\beta|\leq m} C^{m/2}_{\beta\sigma}G_{\frac{m}{2},\lambda,\beta}\ast\{L^{\varepsilon}_{\lambda}[(D^{\sigma}\eta_{k})g]-(D^{\sigma}\eta_{k}L^{\varepsilon}_{\lambda}g)\}
\end{aligned}
\end{equation}
where $g=L^{-m/2}_{\lambda}f$, and $G_{\frac{m}{2},\lambda,\beta}(x)=\lambda^{d/2}D^{\beta}G_{m/2}(\lambda^{1/2}x)$ taken from \eqref{eqn 01.07.15:43} (recall that $1\leq \lambda$). By \cite[Theorem 12.7.1, Theorem 12.7.4]{krylov2008lectures}, $G_{m/2}$ is infinitely differentiable function whose derivatives of any order decay exponentially fast at infinity, and it satisfies
\begin{equation*}
|D^{\beta}G_{m/2}(x)| \leq N \frac{1}{|x|^{d-m+|\beta|}} \quad |\beta|\leq m+1
\end{equation*}
near the origin. Therefore, $G_{\frac{m}{2},\lambda,\beta}$ satisfies
\begin{equation*}
|G_{\frac{m}{2},\lambda,\beta}(x)| = \lambda^{d/2}|D^{\beta}G_{m/2}(\lambda^{1/2}x)| \leq N \frac{1}{|x|^{d}} \quad \forall\,x\in\bR^{d},
\end{equation*}
and
\begin{equation*}
\begin{aligned}
|G_{\frac{m}{2},\lambda,\beta}(x-y)-G_{\frac{m}{2},\lambda,\beta}(x'-y)| &= |x-x'| |\nabla G_{\frac{m}{2},\lambda,\beta}(\theta(x,x',y))| 
\\
&\leq N \lambda^{(d+1)/2}\frac{|x-x'|}{\lambda^{(d+1)/2}|\theta(x,x',y)|^{d+1}}
\\
&\leq N \frac{|x-x'|}{(|x-y|+|x-y'|)^{d+1}}
\end{aligned}
\end{equation*}
whenever $|x-x'|\leq \frac{1}{2}\max{\{|x-y|,|x'-y|\}}$, where $\theta(x,x',y)$ is a point in the line segment connecting $x-y$, and $x'-y$ (recall \eqref{eqn 01.07.16:50}). Therefore, by \cite[Corollary 9.4.7]{grafakos2009modern}, we have
\begin{equation}\label{eqn 01.07.20:34}
\|G_{\frac{m}{2},\lambda,\beta}\ast f\|_{L_{p}(w_{1})} \leq N(d,p,m,[w_{1}]_{p}) \|f\|_{L_{p}(w_{1})}
\end{equation}
for any $f\in L_{p}(w_{1})$. Hence, applying the result for $m=0$ to \eqref{eqn 01.07.16:58} we prove \eqref{eqn 01.07.11:13} for $-m$ when $m$ is an even integer. For $+m$, by \eqref{eqn 01.06.20:02} if we let $g=L^{m/2}_{\lambda}f$,
\begin{equation*}\label{eqn 01.07.16:58-2}
\begin{aligned}
&L^{m/2}_{\lambda}\{L^{\varepsilon}_{\lambda}[\eta_{k}f]-\eta_{k}L^{\varepsilon}_{\lambda}f\}
\\
&\quad = L^{\varepsilon}_{\lambda}(L^{m/2}_{\lambda}[\eta_{k}f])-L^{m/2}_{\lambda}[\eta_{k}L^{\varepsilon}_{\lambda}f]
\\
&\quad \leq L^{\varepsilon}_{\lambda}[\eta_{k}g]-\eta_{k}L^{\varepsilon}_{\lambda}g
\\
&\quad \quad +  \sum_{0<|\sigma|, |\sigma|+|\beta|\leq m} C^{m/2}_{\beta\sigma} \{ L^{\varepsilon}_{\lambda} [(D^{\sigma}\eta_{k})G_{\frac{m}{2},\lambda,\beta}\ast g] - (D^{\sigma}\eta_{k})L^{\varepsilon}_{\lambda}G_{\frac{m}{2},\lambda,\beta} \ast g \}.
\end{aligned}
\end{equation*}
Therefore, by the result for $m=0$ and \eqref{eqn 01.07.20:34}, we have \eqref{eqn 01.07.11:13} for $+m$.

\textit{Case 3.} $m$ is an odd integer.
\\
By \cite[Theorem 3.3]{miller1982sobolev} for any $n$, we have
\begin{equation*}
\begin{gathered}
\|(1-\Delta)^{n}f\|_{L_{p}(w_{1})} \sim \|(1-\Delta)^{n-1/2}f\|_{L_{p}(w_{1})} + \sum_{i=1}^{d} \|(1-\Delta)^{n-1/2}f_{x^{i}}\|_{L_{p}(w_{1})},
\\
\|(1-\Delta)^{n-1/2}f\|_{L_{p}(w_{1})} \leq \|(1-\Delta)^{n}f\|_{L_{p}(w_{1})}.
\end{gathered}
\end{equation*}
Using these and \eqref{eqn 01.08.10:50}, we have
\begin{equation}\label{eqn 01.08.17:38}
\begin{gathered}
\|L^{n}_{\lambda}f\|_{L_{p}(w_{1})} \sim \lambda^{1/2} \|L^{n-1/2}_{\lambda}f\|_{L_{p}(w_{1})} + \sum_{i=1}^{d} \|L^{n-1/2}_{\lambda}f_{x^{i}}\|_{L_{p}(w_{1})},
\\
\|L^{n-1/2}_{\lambda}f\|_{L_{p}(w_{1})} \leq \lambda^{-1/2}\|L^{n}_{\lambda}f\|_{L_{p}(w_{1})}.
\end{gathered}
\end{equation}
Using this and following the argument in \cite[Lemma 2.1 (ii)]{kry94Littlewood}, we have \eqref{eqn 01.07.11:13} when $m$ is an odd integer. 
The lemma is proved.
%we have
%\begin{equation*}
%\begin{aligned}
%&\sum_{k=1}^{\infty}\|L^{\pm m/2}_{\lambda}\{  L^{\varepsilon}_{\lambda}[\eta_{k}f]-\eta_{k}L^{\varepsilon}_{\lambda}f   \}\|^{p}_{L_{p}(w_{1})} 
%\\
%&\quad\leq N \sum_{k=1}^{\infty} \lambda^{p/2} \|L^{\pm m/2-1/2}_{\lambda}\{  L^{\varepsilon}_{\lambda}[\eta_{k}f]-\eta_{k}L^{\varepsilon}_{\lambda}f   \}\|^{p}_{L_{p}(w_{1})}
%\\
%&\quad\quad +N\sum_{k=1}^{\infty}\sum_{i=1}^{d} \|L^{\pm m/2-1/2}_{\lambda}\{  L^{\varepsilon}_{\lambda}[(\eta_{k}f)_{x^{i}}]-(\eta_{k}L^{\varepsilon}_{\lambda}f)_{x^{i}}   \}\|^{p}_{L_{p}(w_{1})}
%\\
%&\quad\leq N \lambda^{\varepsilon p}\|L^{\pm m/2-1/2}f\|^{p}_{L_{p}(w_{1})}
%\\
%&\quad\quad +N\sum_{k=1}^{\infty}\sum_{i=1}^{d} \|L^{\pm m/2-1/2}_{\lambda}\{  L^{\varepsilon}_{\lambda}[\eta_{k}f_{x^{i}}]-(\eta_{k}L^{\varepsilon}_{\lambda}f_{x^{i}})   \}\|^{p}_{L_{p}(w_{1})}
%\\
%&\quad\quad + N\sum_{k=1}^{\infty}\sum_{i=1}^{d} \|L^{\pm m/2-1/2}_{\lambda}\{  L^{\varepsilon}_{\lambda}[\eta_{kx^{i}}f]-(\eta_{kx^{i}}L^{\varepsilon}_{\lambda}f)  \}\|^{p}_{L_{p}(w_{1})}
%\\
%&\quad \leq N\lambda^{(\varepsilon-1/2)p}\|L^{\pm m/2}_{\lambda}f\|^{p}_{L_{p}(w_{1})},
%\end{aligned}
%\end{equation*}
\end{proof}

The following theorem is a generaliization of \cite[Theorem 2.1]{kry94Littlewood}, which treats the case $w_{1}=1$

\begin{theorem}\label{thm 01.25.17:39}
Let $1<p<\infty$, $\gamma\in \bR$, $w_{1}\in A_{p}(\bR^{d})$, and let $\delta>0$. Suppose that $\zeta_{k}\in C^{\infty}(\bR^{d})$ is a sequence of functions satisfying
\begin{equation}
\label{eqn 01.08.15:47}
\sup_{x\in\bR^{d}}\sum_{k=1}^{\infty}|D^{\beta}\zeta_{k}(x)| \leq M(\beta) \quad \text{for any multi-index} \quad \beta.
\end{equation}
Then there exists a constant $N_{0}=N_{0}(d,p,\gamma,[w_{1}]_{p},M)$ such that for any $u\in H^{\gamma}_{p}(w_{1})$
\begin{equation}
\label{eqn 01.09.14:03}
\sum_{k=1}^{\infty}\|\zeta_{k}u\|^{p}_{H^{\gamma}_{p}(w_{1})} \leq N \|u\|^{p}_{H^{\gamma}_{p}(w_{1})}.
\end{equation}
Moreover, if $\zeta_{k}$ satisfies
\begin{equation}
\label{eqn 01.09.13:49}
\delta \leq \sum_{k=1}^{\infty}|\zeta_{k}(x)|^{p},
\end{equation}
then there exists a constant $N_{1}=N_{1}(d,p,\gamma,[w_{1}]_{p},M,\delta)$ such that for any $u\in H^{\gamma}_{p}(w_{1})$
\begin{equation}
\label{eqn 01.08.16:02}
\|u\|^{p}_{H^{\gamma}_{p}(w_{1})} \leq N \sum_{k=1}^{\infty}\|\zeta_{k}u\|^{p}_{H^{\gamma}_{p}(w_{1})}.
\end{equation}
\end{theorem}

\begin{proof}
Due to \eqref{eqn 01.08.10:50} and the definition of $H^{\gamma}_{p}(w_{1})$ it is enough to show that there is  $\lambda> 0$ depending on $d,p,\gamma,[w_{1}]_{p},M,\delta$ such that
\begin{gather}
\label{eqn 01.09.14:08}
 \sum_{k=1}^{\infty}\|L^{\gamma/2}_{\lambda}[\zeta_{k}u]\|^{p}_{L_{p}(w^{\lambda}_{1})} \leq \tilde{N}_{0} \|L^{\gamma/2}_{\lambda}u\|^{p}_{L_{p}(w^{\lambda}_{1})},
\\
\label{eqn 01.08.16:07}
\|L^{\gamma/2}_{\lambda}u\|^{p}_{L_{p}(w^{\lambda}_{1})} \leq \tilde{N}_{1} \sum_{k=1}^{\infty}\|L^{\gamma/2}_{\lambda}[\zeta_{k}u]\|^{p}_{L_{p}(w^{\lambda}_{1})},
\end{gather}
where $w^{\lambda}_{1}(x)=w_{1}(\lambda^{-1/2}x)$, and $\tilde{N}_{0}$, and $\tilde{N}_{1}$ are constants depending only on $d,p,\gamma,[w_{1}]_{p},M$. Moreover, since $\Ccinf(\bR^{d})$ is dense in $H^{\gamma}_{p}(w^{\lambda}_{1})$, we may assume that $u\in \Ccinf(\bR^{d})$. Finally, if we prove the Theorem for $\gamma/2\in \bZ$, then by directly following the proof of \cite[Theorem 2.1]{kry94Littlewood}, we get the desired result for all $\gamma\in \bR$. Hence, we only consider the case $\gamma/2\in \bZ$.

Let $\gamma/2=n\in \bZ$ and let $|n|=m$. If $n\leq 0$, then by \eqref{eqn 01.07.20:34} it follows that
\begin{equation}\label{eqn 01.08.16:57}
\begin{aligned}
&\sum_{k=1}^{\infty}\left\| \sum_{0<|\sigma|, |\sigma|+|\beta|\leq 2m}\lambda^{-|\sigma|/2}G_{m,\lambda,\beta}\ast[D^{\sigma}\zeta_{k}L^{-m}_{\lambda}u] \right\|^{p}_{L_{p}(w^{\lambda}_{1})}
\\
&\leq N \lambda^{-p/2}\sum_{0<|\sigma|,|\sigma|+|\beta|\leq 2m}\sum_{k=1}^{\infty}\| G_{m,\lambda,\beta}\ast[D^{\sigma}\zeta_{k}L^{-m}_{\lambda}u]\|^{p}_{L_{p}(w^{\lambda}_{1})}
\\
&\leq N \lambda^{-p/2}\sum_{|\sigma|\leq 2m}\sum_{k=1}^{\infty} \| D^{\sigma}\zeta_{k}L^{-m}_{\lambda}u\|^{p}_{L_{p}(w^{\lambda}_{1})} \leq N \lambda^{-p/2} \|L^{-m}_{\lambda}u\|^{p}_{L_{p}(w^{\lambda}_{1})}. 
\end{aligned}
\end{equation}
Hence, if we use \eqref{eqn 01.09.13:49} and \eqref{eqn 01.06.20:02-4}, 
\begin{equation}\label{eqn 01.08.17:10}
\begin{aligned}
\|L^{n}_{\lambda}u\|^{p}_{L_{p}(w_{1})} &\leq \delta^{-1}\sum_{k=1}^{\infty}\|\zeta_{k}L^{n}_{\lambda}u\|^{p}_{L_{p}(w^{\lambda}_{1})} 
\\
&\leq N \sum_{k=1}^{\infty} \|L^{-m}_{\lambda}[\zeta_{k}u]\|^{p}_{L_{p}(w^{\lambda}_{1})} 
\\
&\quad + \sum_{k=1}^{\infty}\left\| \sum_{0<|\sigma|, |\sigma|+|\beta|\leq 2m}  C^{m}_{\beta\sigma}\lambda^{-|\sigma|/2}G_{m,\lambda,\beta}\ast[D^{\sigma}\zeta_{k}L^{-m}_{\lambda}u] \right\|^{p}_{L_{p}(w^{\lambda}_{1})}
\\
&\leq N \|L^{n}_{\lambda}[\zeta_{k}u]\|^{p}_{L_{p}(w^{\lambda}_{1})} + N\lambda^{-p/2}\|L^{n}_{\lambda}u\|^{p}_{L_{p}(w^{\lambda}_{1})}.
\end{aligned}
\end{equation}
Here, we emphasize that the constant $N$ does not depend on $\lambda$ since it depends on $[w^{\lambda}_{1}]_{p}=[w_{1}]_{p}$ (recall Remark \ref{rmk 09.27.21:52} (i)). Therefore, if $\lambda\geq 1$ is large enough so that $N\lambda^{-p/2}\leq 1/2$, we have
\begin{equation}\label{eqn 01.08.17:12}
\|L^{n}_{\lambda}u\|^{p}_{L_{p}(w^{\lambda}_{1})} \leq \tilde{N}_{1} \sum_{k=1}^{\infty}\|L^{n}_{\lambda}[\zeta_{k}u]\|^{p}_{L_{p}(w^{\lambda}_{1})},
\end{equation}
and this gives \eqref{eqn 01.08.16:07}. Also from \eqref{eqn 01.06.20:02-4}, \eqref{eqn 01.08.15:47} and \eqref{eqn 01.08.16:57} we have
\begin{equation}\label{eqn 01.08.17:24}
\sum_{k=1}^{\infty}\|L^{n}_{\lambda}[\zeta_{k}u]\|^{p}_{L_{p}(w^{\lambda}_{1})} \leq \tilde{N}_{0} \|L^{n}_{\lambda}u\|^{p}_{L_{p}(w^{\lambda}_{1})},
\end{equation}
where the constant $\tilde{N}_{0}$ does not depend on $\lambda$ and thus we prove \eqref{eqn 01.09.14:08} when $\gamma/2$ is not a nonnegative integer.
If $\gamma/2=n\geq 0$, then by letting $n=m$, like \eqref{eqn 01.08.16:57} we have
\begin{equation*}
\begin{aligned}
\sum_{k=1}^{\infty}\left\|\sum_{0<|\sigma|, |\sigma|+|\beta|\leq 2m} \lambda^{-|\sigma|/2}(D^{\sigma}\zeta_{k})(G_{m,\lambda,\beta}\ast L^{m}_{\lambda}u) \right\|^{p}_{L_{p}(w^{\lambda}_{1})}
\leq N \lambda^{-p/2}\|L^{m}_{\lambda}u\|^{p}_{L_{p}(w^{\lambda}_{1})}.
\end{aligned}
\end{equation*}
This and \eqref{eqn 01.06.20:02} yield
\begin{equation*}
\begin{aligned}
\|L^{n}_{\lambda}u\|^{p}_{L_{p}(w^{\lambda}_{1})} &\leq \delta^{-1}\sum_{k=1}^{\infty}\|\zeta_{k}L^{n}_{\lambda}u\|^{p}_{L_{p}(w^{\lambda}_{1})} 
\\
&\leq N \sum_{k=1}^{\infty} \|L^{m}_{\lambda}[\zeta_{k}u]\|^{p}_{L_{p}(w^{\lambda}_{1})} 
\\
&\quad +\sum_{k=1}^{\infty}\left\|\sum_{0<|\sigma|, |\sigma|+|\beta|\leq 2m} C^{m}_{\beta\sigma}\lambda^{-|\sigma|/2}(D^{\sigma}\zeta_{k})(G_{m,\lambda,\beta}\ast L^{m}_{\lambda}u) \right\|^{p}_{L_{p}(w^{\lambda}_{1})}
\\
&\leq N \|L^{n}_{\lambda}[\zeta_{k}u]\|^{p}_{L_{p}(w^{\lambda}_{1})} + N\lambda^{-p/2}\|L^{n}_{\lambda}u\|^{p}_{L_{p}(w^{\lambda}_{1})}.
\end{aligned}
\end{equation*}
Hence, we have \eqref{eqn 01.08.17:12} for sufficiently large $\lambda \geq 1$, and one can similarly check \eqref{eqn 01.08.17:24} for $\gamma/2=n$ is a nonnegative integer by using \eqref{eqn 01.06.20:02} and \eqref{eqn 01.08.15:47}. The theorem is proved.
\end{proof}

\mysection{Proof of Theorem \ref{theorem 5.1}}

\begin{theorem} \label{theorem 5.1-2}
Let $0<\alpha<2$, $\gamma\in\bR$, $1<p,q<\infty$, $w_{1}=w_{1}(x)\in A_{p}(\R^{d})$, $w_{2}=w_{2}(t)\in A_{q}(\R)$, and $T<\infty$. Suppose that $a^{ij}$ are constants satisfying \eqref{eqn 07.22.2} and $f\in \bH^{\gamma}_{q,p}(w_{2},w_{1},T)$. Then the equation
\begin{equation*}
\partial^{\alpha}_{t}u=a^{ij}u_{x^{i}x^{j}} +f, \quad t>0\,; \quad u(0,\cdot)=1_{\alpha>1}\partial_{t}u(0,\cdot)=0
\end{equation*} 
has a unique solution $u$ in $\bH^{\alpha,\gamma+2}_{q,p,0}(w_{2},w_{1},T)$. Moreover, the solution $u$ satisfies
\begin{equation} \label{eqn 05.08.1-1}
\|u_{xx}\|_{\bH^{\gamma}_{q,p}(w_{2},w_{1},T)}\leq N_0\|f\|_{\bH^{\gamma}_{q,p}(w_{2},w_{1},T)},
\end{equation}
\begin{equation} \label{eqn 05.08.1}
\|u\|_{\bH_{q,p}^{\alpha,\gamma+2}(w_{2},w_{1},T)}\leq N_1\|f\|_{\bH^{\gamma}_{q,p}(w_{2},w_{1},T)},
\end{equation}
where  $N_0=N_0(\alpha,d,\gamma,\delta,p,q,[w_{1}]_{p}, [w_{2}]_{q})$, and $N_1=N_1(\alpha,d,\gamma,\delta,p,q,[w_{1}]_{p}, [w_{2}]_{q},T)$.
\end{theorem}

\begin{proof}
Due to the definition of $\mathbb{H}^{\alpha,\gamma+2}_{q,p,0}(w_{2},w_{1},T)$ and Lemma \ref{lem 05.08.1} (ii) we only consider the case $\gamma=0$. In this case, by \cite[Theorem 2.8]{han19timefractionalAp}, all the claims of theorem are proved when $a^{ij}=\delta^{ij}$. Also by applying Remark \ref{rmk 09.27.21:52} (i), we can prove the theorem for general case (see also \cite[Remark 2.9]{han19timefractionalAp}). 
\end{proof}

\begin{lemma}\label{lem 01.12.12:38}
Let $1<p,q<\infty$, $\gamma\in\bR$, $w_{1}\in A_{p}(\bR^{d})$, $w_{2}\in A_{q}(\bR)$, and let $0<T<\tilde{T}$.

(i) For any $u\in \mathbb{H}^{\alpha,\gamma+2}_{q,p,0}(w_{2},w_{1},T)$, there exists $\tilde{u}\in \mathbb{H}^{\alpha,\gamma+2}_{q,p,0}(w_{2},w_{1},\tilde{T})$ such that $u(t)=\tilde{u}(t)$ for all $t\leq T$, and
\begin{equation}\label{eqn 01.12.13:07}
\|\tilde{u}\|_{\mathbb{H}^{\alpha,\gamma+2}_{q,p,0}(w_{2},w_{1},\tilde{T})} \leq N \|u\|_{\mathbb{H}^{\alpha,\gamma+2}_{q,p,0}(w_{2},w_{1},T)},
\end{equation} 
where the constant $N$ is independent of $T$.

(ii) Let $u,v\in \mathbb{H}^{\alpha,\gamma+2}_{q,p,0}(w_{2},w_{1},\tilde{T})$ and $u(t)=v(t)$ for $t\leq T$. Then $\bar{u}(t):=u(T+t)-v(T+t)\in \mathbb{H}^{\alpha,\gamma+2}_{q,p,0}(w_{2,T},w_{1},\tilde{T}-T)$, where $w_{2,T}(t)=w_{2}(T+t)$.
\end{lemma}

\begin{proof}
By Lemma \ref{lem 05.08.1} (ii), we may assume that $\gamma=0$.

(i) By Theorem \ref{theorem 5.1-2}, the equation
$$
\partial^{\alpha}_{t}\tilde{u}=\Delta \tilde{u}+(\partial^{\alpha}_{t}u-\Delta u)1_{t\leq T} ,\quad t\leq \tilde{T}
$$
has a unique solution $\tilde{u}\in \mathbb{H}^{\alpha,2}_{q,p,0}(w_{2},w_{1},\tilde{T})$, and \eqref{eqn 01.12.13:07} also follows. Moreover, since $\bar{u}=u-\tilde{u}$ satisfies 
$$
\partial^{\alpha}_{t}\bar{u}=\Delta \bar{u} \quad t\leq T,
$$
we have $u=\tilde{u}$ for $t\leq T$ due to the uniqueness of solution.

(ii) Take sequences $u_{n},v_{n}\in \Ccinf((0,\infty)\times\bR^{d})$ such that $u_{n}$ and $v_{n}$ converges to $u$ and $v$ in $\mathbb{H}^{\alpha,2}_{q,p,0}(w_{2},w_{1},\tilde{T})$ respectively. Note that by Remark \ref{rmk 01.14.13:13}, one can take $u_{n},v_{n}$ so that
$$
u_{n}(t,x)=v_{n}(t,x)\quad \forall t\leq (T+a_{n})\wedge \tilde{T}.
$$
Therefore, if we define
$$
\bar{u}_{n}=u_{n}(T+t,x)-v_{n}(T+t,x),
$$
we have $\bar{u}_{n}\in \Ccinf((0,\infty)\times\bR^{d})$, and 
$$
\bar{u}_{n}(0,x)=0,\quad 1_{\alpha>1}\partial_{t}\bar{u}_{n}(0,x)=0 \quad \forall x\in \bR^{d}.
$$
Moreover, since $\bar{u}_{n}$ converges to $\bar{u}$ in $\mathbb{H}^{\alpha,\gamma+2}_{q,p}(w_{2,T},w_{1},\tilde{T}-T)$, we have the desired result. The lemma is proved.
\end{proof}

\begin{lemma}\label{lem 01.12.17:45}
Let $0<\alpha<2$, $1<p,q<\infty$, $\gamma\in\bR$, $w_{1}=w_{1}(x)\in A_{p}(\R^{d})$, $w_{2}=w_{2}(t)\in A_{q}(\R)$, and $T<\infty$. Suppose that Assumption \ref{asm 07.16.1} holds and $f\in \bH^{\gamma}_{q,p}(w_{2},w_{1},T)$. Also suppose that $u\in\mathbb{H}^{\alpha,\gamma+2}_{q,p,0}(w_{2},w_{1},T)$ is a solution for \eqref{eqn 05.09.1} satisfying
\begin{equation}\label{eqn 01.12.17:47}
\|u\|_{\mathbb{H}^{\gamma+2}_{q,p}(w_{2},w_{1},t)}\leq N_{0} \left( \|f\|_{\mathbb{H}^{\gamma}_{q,p}(w_{2},w_{1},t)} + \|u\|_{\mathbb{H}^{\gamma+1}_{q,p}(w_{2},w_{1},t)}  \right) \quad \forall\,t\leq T_{0},
\end{equation}
where $T_{0}\leq T$, and the constant $N_{0}$ is independent of $u,f$, and $t$. Then there exists $t_{0}<T_{0}$ depending only on $\alpha,d,p,q,\gamma,\delta,[w_{1}]_{p},[w_{2}]_{q}$, and $N_{0}$ such that
\begin{equation}\label{eqn 01.12.17:47-2}
\|u\|_{\mathbb{H}^{\gamma+2}_{q,p}(w_{2},w_{1},t)}\leq N \|f\|_{\mathbb{H}^{\gamma}_{q,p}(w_{2},w_{1},t)} \quad \forall\, t<t_{0},
\end{equation}
where the constant $N$ depends only on $\alpha,d,p,q,\gamma,\delta,[w_{1}]_{p},[w_{2}]_{q}$, and $N_{0}$.
\end{lemma}

\begin{proof}
From \eqref{eqn 01.12.17:47} and \eqref{eqn 01.14.16:41} for any $t\leq T$ and $\varepsilon>0$ we have
\begin{equation*}
\begin{aligned}
\|u\|_{\mathbb{H}^{\gamma+2}_{q,p}(w_{2},w_{1},t)} &\leq N_{0} \|f\|_{\mathbb{H}^{\gamma}_{q,p}(w_{2},w_{1},t)} + \varepsilon N_{0}\|u\|_{\mathbb{H}^{\gamma+2}_{q,p}(w_{2},w_{1},t)} 
\\
&\quad + N(N_{0},\varepsilon)\|u\|_{\mathbb{H}^{\gamma}_{q,p}(w_{2},w_{1},t)}.
\end{aligned}
\end{equation*}
Hence, if we take $\varepsilon>0$ small enough such that $\varepsilon N_{0} <1/2$, we have
$$
\|u\|_{\mathbb{H}^{\gamma+2}_{q,p}(w_{2},w_{1},t)} \leq 2N_{0} \|f\|_{\mathbb{H}^{\gamma}_{q,p}(w_{2},w_{1},t)} + N(N_{0})\|u\|_{\mathbb{H}^{\gamma}_{q,p}(w_{2},w_{1},t)}.
$$
for $t\leq T_{0}$. Since $u$ satisfies \eqref{eqn 05.09.1}, by \eqref{eqn 12.21.16:05} and Lemma \ref{lem 01.06.16:05} we have
\begin{equation*}
\begin{aligned}
\|u\|_{\mathbb{H}^{\gamma+2}_{q,p}(w_{2},w_{1},t)} &\leq 2N_{0} \|f\|_{\mathbb{H}^{\gamma}_{q,p}(w_{2},w_{1},t)} + N(N_{0})\|u\|_{\mathbb{H}^{\gamma}_{q,p}(w_{2},w_{1},t)}.
\\
& \leq  N_{1}t^{\alpha}(\|u\|_{\mathbb{H}^{\gamma+2}_{q,p}(w_{2},w_{1},t)} + \|f\|_{\mathbb{H}^{\gamma}_{q,p}(w_{2},w_{1},t)})
\\
& \quad + 2N_{0}\|f\|_{\mathbb{H}^{\gamma}_{q,p}(w_{2},w_{1},t)},
\end{aligned}
\end{equation*}
where $N_{1}=N_{1}(\alpha,d,p,q,\gamma,\delta,[w_{1}]_{p},[w_{2}]_{q},N_{0})$. Therefore, if we take $t_{0}$ such that $N_{1}(t_{0})^{\alpha}<1/4$, we have
$$
\|u\|_{\mathbb{H}^{\gamma+2}_{q,p}(w_{2},w_{1},t)}\leq N \|f\|_{\mathbb{H}^{\gamma}_{q,p}(w_{2},w_{1},t)}
$$
for all $t<t_{0}$, where $N=N(\alpha,d,p,q,\gamma,\delta,[w_{1}]_{p},[w_{2}]_{q},N_{0})$. The lemma is proved.
\end{proof}

\textbf{Proof of Theorem \ref{theorem 5.1}}
\\
Due to Theorem \ref{theorem 5.1-2} we only prove the a priori estimate. Moreover, by the definition of the norm $\|u\|_{\mathbb{H}^{\alpha,\gamma+2}_{q,p}(w_{2},w_{1},T)}$ we only prove the a priori estimate with $\|u\|_{\mathbb{H}^{\gamma+2}_{q,p}(w_{2},w_{1},T)}$ in place of $\|u\|_{\mathbb{H}^{\alpha,\gamma+2}_{q,p}(w_{2},w_{1},T)}$.

\textit{Step 1.}  $p=q$,  $a^{ij}=a^{ij}(t)$ are independent of $x$ and $b^{i}=c=0$.
\\
Suppose that $u\in\mathbb{H}^{\alpha,\gamma+2}_{p,p,0}(w_{2},w_{1},T)$ satisfies \eqref{eqn 05.09.1}. If we let $a^{ij}_{0}=a^{ij}(0)$, then $u$ satisfies
$$
\partial^{\alpha}_{t}u=a^{ij}_{0}u_{x^{i}x^{j}}+f+(a^{ij}-a^{ij}_{0})u_{x^{i}x^{j}}.
$$
Hence, for each $t\leq T$, we have
\begin{equation*}
\begin{aligned}
\|u\|_{\mathbb{H}^{\gamma+2}_{p,p}(w_{2},w_{1},t)} &\leq N \left( \|f\|_{\mathbb{H}^{\gamma}_{p,p}(w_{2},w_{1},t)} + \|(a^{ij}-a^{ij}_{0})D^{2}u\|_{\mathbb{H}^{\gamma}_{p,p}(w_{2},w_{1},t)} \right)
\\
& = N\left( \|f\|_{\mathbb{H}^{\gamma}_{p,p}(w_{2},w_{1},t)}+ \|(a^{ij}-a^{ij}_{0})(1-\Delta)^{\gamma/2}D^{2}u\|_{\mathbb{L}_{p,p}(w_{2},w_{1},t)} \right)
\\
&\leq N\left( \|f\|_{\mathbb{H}^{\gamma}_{p,p}(w_{2},w_{1},t)}+\sup_{0<s<t} |a^{ij}(s)-a^{ij}_{0}| \|u\|_{\mathbb{H}^{\gamma+2}_{p,p}(w_{2},w_{1},t)} \right),
\end{aligned}
\end{equation*}
where the constant $N$ is independent of $t$. Due to the uniform continuity of $a^{ij}$ there exists a $T_{0}$ such that if $t<T_{0}$, then the $N\sup_{0<s<t} |a^{ij}(s)-a^{ij}_{0}|<1/2$. Therefore, by Lemma \ref{lem 01.12.17:45} there is $t_{0}\leq T_{0}$ such that
\begin{equation}\label{eqn 01.13.15:34}
\|u\|_{\mathbb{H}^{\gamma+2}_{p,p}(w_{2},w_{1},t)} \leq N \|f\|_{\mathbb{H}^{\gamma}_{p,p}(w_{2},w_{1},t)}
\end{equation}
for all $t<t_{0}$. Take an integer $l$ such that $T/l<\frac{1}{2}t_{0}$ and let $T_{k}=kT/l$. Suppose that for $k\leq l$ we have \eqref{eqn 01.13.15:34} with $T_{k}$ in place of $t$. Take $\tilde{u}$ from Lemma \ref{lem 01.12.12:38} (i) corresponding to $T=T_{k}$ and $\tilde{T}=T_{k+1}$. Set $\bar{u}(t)=(u-\tilde{u})(T_{k}+t)$. Then $\bar{u}$ satisfies
$$
\partial^{\alpha}_{t}\bar{u}=a^{ij}\bar{u}_{x^{i}x^{j}}+\bar{f} \quad t\leq T_{k+1}-T_{k}
$$
where
$$
\bar{f}(t,x)=f(T_{k}+t,x)-\partial^{\alpha}_{t}\tilde{u}(T_{k}+t,x)+a^{ij}\tilde{u}(T_{k}+t,x).
$$
Therefore, by the above result and Lemma \ref{lem 01.12.17:45}, we have
\begin{equation*}
\begin{aligned}
\|\bar{u}\|_{\mathbb{H}^{\gamma+2}_{p,p}(w_{2,T_{k}},w_{1},T_{1})} &\leq N \|\bar{f}\|_{\mathbb{H}^{\gamma}_{p,p}(w_{2,T_{k}},w_{1},T_{1})}
\\
& \leq N \left( \|f\|_{\mathbb{H}^{\gamma}_{p,p}(w_{2},w_{1},T_{k+1})} + \|\tilde{u}\|_{\mathbb{H}^{\gamma+2}_{p,p}(w_{2},w_{1},T_{k+1})} \right)
\\
&\leq N \left( \|f\|_{\mathbb{H}^{\gamma}_{p,p}(w_{2},w_{1},T_{k+1})}+ \|u\|_{\mathbb{H}^{\gamma+2}_{p,p}(w_{2},w_{1},T_{k})} \right)
\\
&\leq N \|f\|_{\mathbb{H}^{\gamma}_{p,p}(w_{2},w_{1},T_{k+1})},
\end{aligned}
\end{equation*}
where $w_{2,T_{k}}(t)=w_{2}(T_{k}+t)$ and the last inequality holds due to the assumption that the a priori estimate holds for $T_{k}$. Hence, by \eqref{eqn 01.12.13:07} we have
\begin{equation*}
\begin{aligned}
\|u\|_{\mathbb{H}^{\gamma+2}_{p,p}(w_{2},w_{1},T_{k+1})} &\leq \|u\|_{\mathbb{H}^{\gamma+2}_{p,p}(w_{2},w_{1},T_{k})} + \|u(T_{k}+\cdot)\|_{\mathbb{H}^{\gamma+2}_{p,p}(w_{2,T_{k}},w_{1},T_{k+1}-T_{k})}
\\
&\leq N \|f\|_{\mathbb{H}^{\gamma}_{p,p}(w_{2},w_{1},T_{k})}+ \|\tilde{u}\|_{\mathbb{H}^{\gamma+2}_{p,p}(w_{2},w_{1},T_{k+1})} + \|\bar{u}\|_{\mathbb{H}^{\gamma+2}_{p,p}(w_{2,T_{k}},w_{1},T_{1})}
\\
&\leq N \|f\|_{\mathbb{H}^{\gamma}_{p,p}(w_{2},w_{1},T_{k+1})}.
\end{aligned}
\end{equation*}
Therefore, by induction we prove the a priori estimate.

\textit{Step 2.} $p=q$, $a^{ij}=a^{ij}(t,x)$ and $b^{i}=c=0$.
\\
Suppose that $u\in \mathbb{H}^{\alpha,\gamma+2}_{p,p,0}(w_{2},w_{1},T)$ satisfies \eqref{eqn 05.09.1}.   For $r>0$ let 
$$
a^{ij}_{r}(t,x)=a^{ij}(r^{2/\alpha}t,rx),\quad u_{r}(t,x)=u(r^{2/\alpha}t,rx), \quad f_{r}(t,x)=r^{2}f(r^{2/\alpha}t,rx).
$$
Then  we have
$$
\partial^{\alpha}_{t}u_{r}=a^{ij}_{r0}(u_{r})_{x^{i}x^{j}}+ f_{r} +(a^{ij}_{r}-a^{ij}_{0})(u_{r})_{x^{i}x^{j}} \quad  t\leq r^{-2/\alpha}T,
$$
where $a^{ij}_{r0}(t,x)=a^{ij}_{r}(t,0)$. By the above result and Lemma \ref{lem 01.06.16:05}, for any $t\leq r^{-2/\alpha}T$,
\begin{equation*}
\begin{aligned}
\|(u_{r})_{xx}\|_{\mathbb{H}^{\gamma}_{p,p}(w^{r}_{2},w^{r}_{1},t)} &\leq N_{0} \|f_{r}\|_{\mathbb{H}^{\gamma}_{p,p}(w^{r}_{2},w^{r}_{1},t)} 
\\
&\quad +N_{0} \sup_{t} |a^{ij}_{r0}-a^{ij}_{r}|_{B^{|\gamma|}}\|(u_{r})_{xx}\|_{\mathbb{H}^{\gamma}_{p,p}(w^{r}_{2},w^{r}_{1},t)},
\end{aligned}
\end{equation*}
where $w^{r}_{2}(t)=w_{2}(r^{2/\alpha}t)$ $w^{r}_{1}(x)=w(rx)$, and the constant $N_{0}$ is independent of $r$ and $t$. Note that due to the definition of $B^{|\gamma|}$, we have
$$
\sup_{t}|a^{ij}_{r0}(t,\cdot)-a^{ij}_{r}(t,\cdot)|_{B^{|\gamma|}} \leq \sup_{(t,x)}|a^{ij}(t,x)-a^{ij}(t,0)| + 1_{\gamma\neq 0} r^{|\gamma|\wedge 1}\delta^{-1}.
$$
Therefore, if there is a $\varepsilon_{1}\in(0,1)$ and $r$ small enough so that 
$$
\sup_{(t,x)}|a^{ij}(t,x)-a^{ij}(t,0)| + 1_{\gamma\neq 0} r^{|\gamma|\wedge 1}\delta^{-1}  \leq \varepsilon_{1}+ r^{|\gamma|\wedge 1}\delta^{-1} < \frac{1}{2}N_{0},
$$
for each $t\leq r^{-2/\alpha}T$, we have
$$
\|(u_{r})_{xx}\|_{\mathbb{H}^{\gamma}_{p,p}(w^{r}_{2},w^{r}_{1},t)}\leq 2N_{0}  \|f_{r}\|_{\mathbb{H}^{\gamma}_{p,p}(w^{r}_{2},w^{r}_{1},t)},
$$
and by change of variables we have
$$
\|u_{xx}\|_{\mathbb{H}^{\gamma}_{p,p}(w_{2},w_{1},t)}\leq 2N_{0}  \|f\|_{\mathbb{H}^{\gamma}_{p,p}(w_{2},w_{1},t)}
$$
for any $t\leq T$. From this and \eqref{eqn 12.21.16:05}, for any $t\leq T$, we have
\begin{equation}\label{eqn 01.12.17:02}
\|u\|_{\mathbb{H}^{\gamma+2}_{p,p}(w_{2},w_{1},t)} \leq N \|f\|_{\mathbb{H}^{\gamma}_{p,p}(w_{2},w_{1},t)}
\end{equation}
provided that $\sup_{(t,x)}|a^{ij}(t,x)-a^{ij}(t,0)|\leq \varepsilon_{1}$, where $N$ is independent of $t$.

Now take $\delta_{1}\in (0,1)$ so that $\sup_{t}|a^{ij}(t,x)-a^{ij}(t,y)|\leq \varepsilon_{1}/2$ whenever $|x-y|\leq 4\delta_{1}$. Also for this $\delta_{1}$, take a partition of unity $\zeta_{k}\in \Ccinf(\bR^{d})$ so that $0\leq \zeta_{k}\leq 1$, and support of $\zeta_{k}$ is contained in $B_{\delta_{1}}(x_{k})$ for some $x_{k}\in\bR^{d}$ and satisfying
\begin{gather*}
\sup_{x\in\bR^{d}}\sum_{k=1}^{\infty}|D^{\sigma}\zeta_{k}(x)| \leq M(\sigma) <\infty,
\\
0<1=\sum_{k=1}^{\infty}|\zeta_{k}(x)| \quad \forall x\in\bR^{d}.
\end{gather*}
Also take $\eta\in \Ccinf(\bR^{d})$ so that $0\leq \eta\leq 1$, $\eta=1$ on $B_{1}$, and $\eta=0$ outside of $B_{2}$. Define $\eta_{k}=\eta((x-x_{k})/\delta_{1})$ and $u_{k}=u\zeta_{k}$. Then if we define
\begin{equation}\label{eqn 01.14.10:21}
a^{ij}_{k}(t,x)=\eta_{k}(x)a^{ij}(t,x) + (1-\eta_{k}(x))a^{ij}(t,x_{k}),
\end{equation}
it follows that
$$
\partial^{\alpha}_{t}u_{k}=a^{ij}_{k}(u_{k})_{x^{i}x^{j}}+f_{k},
$$
where
$$
f_{k}=f\zeta_{k}+(a^{ij}_{k}u_{x^{i}x^{j}}\zeta_{k}-a^{ij}_{k}(u_{k})_{x^{i}x^{j}})=f\zeta_{k}-a^{ij}(2u_{x^{i}}(\zeta_{k})_{x^{j}}+u(\zeta_{k})_{x^{i}x^{j}})
$$
since $\eta_{k}=1$ on the support of $\zeta_{k}$. Also note that for any $t$ and $x,y\in\bR^{d}$,
\begin{equation}\label{eqn 01.14.10:23}
\begin{aligned}
|a^{ij}_{k}(t,x)-a^{ij}_{k}(t,y)| &=|\eta_{k}(x)(a^{ij}(t,x)-a^{ij}(t,x_{k}))-\eta_{k}(y)(a^{ij}(t,y)-a^{ij}(t,x_{k}))|
\\
&\leq |\eta_{k}(x)(a^{ij}(t,x)-a^{ij}(t,x_{k}))| 
\\
&\quad\quad+ |\eta_{k}(y)(a^{ij}(t,y)-a^{ij}(t,x_{k}))| \leq \varepsilon_{1}.
\end{aligned}
\end{equation}
Since $a^{ij}_{k}$ satisfies \eqref{eqn 07.22.2}, for each $t\leq T$ we have
\begin{equation}\label{eqn 01.14.10:46}
\begin{aligned}
&\|u\|^{p}_{\mathbb{H}^{\gamma+2}_{p,p}(w_{2},w_{1},t)}
\\
& \leq N \sum_{k=1}^{\infty}\|u_{k}\|^{p}_{\mathbb{H}^{\gamma+2}_{p,p}(w_{2},w_{1},t)} \leq N\sum_{k=1}^{\infty}\|f_{k}\|^{p}_{\mathbb{H}^{\gamma}_{p,p}(w_{2},w_{1},t)}
\\
&\leq N \sum_{k=1}^{\infty}\left(  \|f\zeta_{k}\|^{p}_{\mathbb{H}^{\gamma}_{p,p}(w_{2},w_{1},t)} + \|DuD\zeta_{k}\|^{p}_{\mathbb{H}^{\gamma}_{p,p}(w_{2},w_{1},t)} + \|uD^{2}\zeta_{k}\|^{p}_{\mathbb{H}^{\gamma}_{p,p}(w_{2},w_{1},t)} \right)
\\
&\leq N \|f\|^{p}_{\mathbb{H}^{\gamma}_{p,p}(w_{2},w_{1},t)} + N \|u\|^{p}_{\mathbb{H}^{\gamma+1}_{p,p}(w_{2},w_{1},t)}
\end{aligned}
\end{equation}
by Theorem \ref{thm 01.25.17:39} and \eqref{eqn 01.12.17:02}, where the constant $N$ does not depend on $t$. Therefore, by taking $t_{0}$ from Lemma \ref{lem 01.12.17:45}, for $t<t_{0}$ we have
\begin{equation*}
\|u\|_{\mathbb{H}^{\gamma+2}_{p,p}(w_{2},w_{1},t)} \leq N \|f\|_{\mathbb{H}^{\gamma}_{p,p}(w_{2},w_{1},t)}.
\end{equation*}
By following the induction argument in Step 1, we prove the a priori estimate.

\textit{Step 3.} $p\neq q$, and $b^{i}=c=0$.
\\
By the assumption on $a^{ij}$, there exists $R>0$ such that the oscillation of $a^{ij}$ outside of $B_{R/2}$ is less than $\varepsilon_{1}/2$. For this $R>0$, take a number $M$ such that $\sum_{k=1}^{M}\zeta_{k}=1$ on $B_{R}$ and vanishes outside of $B_{2R}$. Denote $\zeta_{0}=1-\sum_{k=1}^{M}\zeta_{k}$. If we define $a^{ij}_{k}$ as in \eqref{eqn 01.14.10:21} for $k=0,1,\dots,M$, then they satisfy \eqref{eqn 01.14.10:23}. Indeed, since the oscillation of $a^{ij}$ is bounded by $\varepsilon_{1}/2$ in the support of $\zeta_{0}$, by taking $\eta_{0}\in C^{\infty}(\bR^{d})$ such that $\eta_{0}=0$ on $B_{R/2}$ and $\eta_{0}=1$ outside of $B_{R}$, and defining
$$
a^{ij}_{0}(t,x)=\eta_{0}(x)a^{ij}(t,x) + (1-\eta_{k}(x))a^{ij}(t,x_{0}),
$$
where $|x_{0}|\geq R$, we have \eqref{eqn 01.14.10:23} for $a^{ij}_{0}$. By following \eqref{eqn 01.14.10:46}, we have
\begin{equation*}
\begin{aligned}
\|u\|^{q}_{\mathbb{H}^{\gamma+2}_{q,p}(w_{2},w_{1},t)} &\leq N\sum_{k=0}^{M}\|u\zeta_{k}\|^{q}_{\mathbb{H}^{\gamma+2}_{q,p}(w_{2},w_{1},t)}
\\
&\leq N \sum_{k=0}^{M}\big(   \|f\zeta_{k}\|^{q}_{\mathbb{H}^{\gamma}_{q,p}(w_{2},w_{1},t)} + \|DuD\zeta_{k}\|^{q}_{\mathbb{H}^{\gamma}_{q,p}(w_{2},w_{1},t)} 
\\
&\quad\quad\quad\quad \quad+ \|uD^{2}\zeta_{k}\|^{q}_{\mathbb{H}^{\gamma}_{q,p}(w_{2},w_{1},t)} \big)
\\
&\leq N \left( \|f\|^{q}_{\mathbb{H}^{\gamma}_{q,p}(w_{2},w_{1},t)} + \|u\|^{q}_{\mathbb{H}^{\gamma+1}_{q,p}(w_{2},w_{1},t)} \right).
\end{aligned}
\end{equation*}
Hence, by following the argument after \eqref{eqn 01.14.10:46}, we prove the a priori estimate.

\textit{Step 4.} General case.
\\
Suppose that $u\in\mathbb{H}^{\alpha,\gamma+2}_{q,p,0}(w_{2},w_{1},T)$ satisfies \eqref{eqn 05.09.1}. Then by Step 3, for any $t\leq T$ we have
\begin{equation*}
\begin{aligned}
\|u\|_{\mathbb{H}^{\gamma+2}_{q,p}(w_{2},w_{1},t)} &\leq N \left( \|f\|_{\mathbb{H}^{\gamma}_{q,p}(w_{2},w_{1},t)} +\|b^{i}u_{x^{i}}+cu\|_{\mathbb{H}^{\gamma}_{q,p}(w_{2},w_{1},t)}  \right)
\\
&\leq  N \left( \|f\|_{\mathbb{H}^{\gamma}_{q,p}(w_{2},w_{1},t)} +\|u\|_{\mathbb{H}^{\gamma+1}_{q,p}(w_{2},w_{1},t)}   \right),
\end{aligned}
\end{equation*}
where the constant $N$ does not depend on $t$. Therefore, by Lemma \ref{lem 01.12.17:45}, there exists $t_{0}<T$ such that
$$
\|u\|_{\mathbb{H}^{\gamma+2}_{p,p}(w_{2},w_{1},t)}\leq N \|f\|_{\mathbb{H}^{\gamma}_{p,p}(w_{2},w_{1},t)} \quad \forall\, t<t_{0}.
$$
Therefore, by following the induction argument in Step 1, we prove the a priori estimate. The theorem is proved.

%\bibliography{refs}
%\bibliographystyle{plain}
\end{document}